\documentclass{amsart}
\usepackage{amsmath,amssymb,amsfonts,mathrsfs}
\usepackage{amsthm} 
\usepackage{varioref}
\usepackage[svgnames]{xcolor}
\usepackage{tikz}
\usepackage{tkz-graph}
\usepackage{xfrac}
\usepackage[labelformat=simple]{subcaption}
\usepackage{enumerate}
\usepackage{etoolbox}
\usetikzlibrary{patterns,calc,arrows,positioning,decorations.pathreplacing}

\tikzstyle{underbrace style}=[decorate,decoration={brace,raise=3mm,amplitude=3pt,mirror}]
\tikzstyle{underbrace text style}=[scale=0.8, below, pos=.5, yshift=-7mm]
\tikzstyle{overbrace style}=[decorate,decoration={brace,raise=6mm,amplitude=3pt}]
\tikzstyle{overbrace text style}=[scale=0.8, below, pos=.5, yshift=18mm]

\newtoggle{COLORS}
\newtoggle{NOTES}
\togglefalse{COLORS}
\togglefalse{NOTES}

\iftoggle{COLORS}{
  \def\trajColor{red}
}{
  \def\trajColor{gray}
}

\renewcommand{\epsilon}{\varepsilon}

\def\R{\mathbb{R}}
\def\T{\mathbb{T}}
\def\Z{\mathbb{Z}}
\def\C{\mathbb{C}}

\renewcommand{\phi}{\varphi}

\DeclareMathOperator{\lcm}{lcm}
\DeclareMathOperator{\GL}{GL}
\DeclareMathOperator{\SL}{SL}
\DeclareMathOperator{\id}{id}

\newtheorem{theorem}{Theorem}
\newtheorem{proposition}[theorem]{Proposition}
\newtheorem{lemma}[theorem]{Lemma}
\newtheorem{corollary}[theorem]{Corollary}
\theoremstyle{definition}
\newtheorem*{remark}{Remark}

\title{Cutting Sequences on Square-Tiled Surfaces}
\author{Charles C. Johnson}
\date{\today}

\thanks{The author thanks the referee for their careful reading of the
  original manuscript and for their many helpful remarks for improving
  the exposition.}


\begin{document}
\maketitle

\begin{abstract}
  We characterize cutting sequences of infinite geodesics on
  square-tiled surfaces by considering interval exchanges on specially
  chosen intervals on the surface.  These interval exchanges can be
  thought of as skew products over a rotation, and we convert cutting
  sequences to symbolic trajectories of these interval exchanges to
  show that special types of \emph{combinatorial lifts} of Sturmian
  sequences completely describe all cutting sequences on a
  square-tiled surface.  Our results extend the list of families of
  surfaces where cutting sequences are understood to a dense subset of
  the moduli space of all translation surfaces.
\end{abstract}

\section{Introduction}
Given a surface with a Riemannian metric and a labeled collection of
curves on the surface, the \emph{cutting sequence} of a geodesic is a
biinfinite sequence of labels given in the order in which the geodesic
intersects the corresponding curves.  The most obvious question to
consider concerns characterizing precisely which sequences of labels
correspond to cutting sequences.

This question has been studied for a few special types of translation
surfaces, which are surfaces equipped with a flat metric with
cone-type singularities and trivial holonomy.  In particular, cutting
sequences have been described for flat tori \cite{HedlundMorse}
\cite{Series}; the regular polygon surfaces with an even number of
sides were studied by Smillie and Ulcigrai, \cite{SmillieUlcigrai};
Davis later extended the methods of
Smillie and Ulcigrai to double $n$-gon surfaces with odd $n$ \cite{DavisRegular};
the $L$-shaped surface built from three squares was studied by Wu and
Zhong, \cite{WuZhong}; and most recently the family of Bouw-M\"oller
surfaces was studied by Davis, Pasquinelli, and Ulcigrai \cite{DPU}.
Most of the previous results are obtained by considering the action of
an element of the Veech group (group of affine symmetries) of the
surface.  In \cite{DavisPerfect}, for example, the effect of
\emph{flip-shears} on cutting sequences on some special types of
surfaces are considered.  In this paper we add one more family of
surfaces to this list by characterizing cutting sequences on
square-tiled surfaces.

Our approach will be to convert the question of which sequences are
cutting sequences into a question about the languages determined by a
special class of interval exchange transformations which can be
described as a certain type of skew product over a rotation.
Languages of the family of interval exchanges with a fixed permutation
were studied in \cite{FerencziZamboni}. 



After recalling the pertinent background and establishing some
vocabulary and notation we will prove that all cutting sequences on a
given square-tiled surface can be described as sequences of pairs with
the first entry in each pair giving the label of some square on the
surface, and the second entry giving an edge of the square.  The
second entries of these pairs must form a cutting sequence on a torus,
and the first entries must form a sequence of labels that is
consistent with the way the various squares on the surface are glued
together.  The sequences of pairs satisfying these requirements have a
relatively simple combinatorial description, but not all such
sequences are cutting sequences on the surface as this collection
includes sequences which would describe a geodesic which hits a
conical singularity on the surface.  We will show there is a
combinatorial description of these sequences by considering walks on a
family of graphs associated to the surface.

We describe the necessary background on square tiled surfaces,
Sturmian sequences, languages, and interval exchanges in
Section~\ref{sec:background}.  We then show in Section~\ref{sec:sqiet}
that interval exchanges for some specially chosen intervals on a
square-tiled surface are simply skew products over rotations which
greatly simplifies understanding these interval exchanges.  In
Section~\ref{sec:consistency} we introduce the notion of consistency
between sequences of edges and the gluings of the surface and describe
symmetric Sturmian sequences.  The complete statement and proof of the
characterization theorem appears in Section~\ref{sec:proof}.

\section{Background}
\label{sec:background}
\subsection{Translation surfaces and the Veech dichotomy}
A \emph{translation surface} is a closed, oriented surface equipped
with a geometry which is flat away from a discrete set of
\emph{conical singularities} where small neighborhoods of a
singularity are isometric to a Euclidean cone of cone angle $2 \pi n$
with $n \geq 1$ an integer.  Such a surface can always be represented
as a collection of polygons in the plane together with gluings that
identify the edges of polygons in pairs by translations where the
inward-pointing normals of two identified edges point in opposite
directions.  Equivalently, a translation surface is a pair
$(X, \omega)$ where $X$ is a compact Riemann surface and $\omega$ is a
holomorphic 1-form on $X$.  The equivalence between these two
definitions comes from endowing $(X, \omega)$ with a singular flat metric
by integrating $\omega$ to obtain an atlas of charts where transitions
are given by translations (hence the name \emph{translation surface}),
and so the usual Euclidean metric of $\C$ can be pulled back to a
metric on the surface.  Choosing a geodesic triangulation of the
surface where all singularities occur as vertices of the triangulation
then gives a description of $(X, \omega)$ as a collection of polygons
with edge gluings.

The set of all translation surfaces where the 1-form $\omega$ has the
orders of its zeros fixed is called a \emph{stratum} and is equipped
with the structure of a complex orbifold by assigning local
coordinates by integrating $\omega$ along the relative cycles of each
surface's first homology $H_1(X, \{p_1, p_2, ..., p_n\}; \C)$,
relative with respect to the zeros of the 1-form.  The group
$\GL(2, \R)$ acts on each stratum by deforming the polygonal
representation of a surface.  Since these deformations are linear,
parallel lines are preserved and so the deformed polygon still glues
together to give a translation surface in the same stratum.

The stabilizer of a surface $(X, \omega)$ under the $\GL(2, \R)$
action is called the \emph{Veech group} of the surface and is denoted
$\SL(X, \omega)$.  It was shown in \cite{Veech89} that when
$\SL(X, \omega)$ is a lattice, then for every direction the geodesic
flow on $(X, \omega)$ satisfies the following dichotomy: the flow is
either completely periodic or uniquely ergodic.  It had been
previously shown in \cite{KMS} that for every translation surface the
flow in almost every direction is uniquely ergodic, but there are
surfaces where the flow in some directions is minimal but not uniquely
ergodic.  \emph{Veech surfaces}, the surfaces with a lattice
stabilizer, are surfaces where this is impossible: every non-periodic
direction is uniquely ergodic.

More information about translation surfaces can be found in the
surveys \cite{MasurTabachnikov}, \cite{Zorich}, and \cite{Wright}.

\subsection{Square-tiled surfaces}
Let $\T^2$ denote the standard square torus, $\C / \Z[i]$ equipped
with the $1$-form $dz$.  Notice that the torus is a Veech surface with
Veech group $\SL(2, \Z)$.  A \emph{square-tiled surface} is a covering
$\pi : X \to \T^2$ branched over at most one point.  The cover $X$ is
then given the translation structure corresponding to
$\omega = \pi^*(dz)$.  We can describe $X$ as a collection of unit
squares, the number of squares being the degree of the covering, glued
together so that the top of each square is identified with the bottom
of another (or the same) square, and similarly for left- and
right-hand edges of the squares.  The gluings of the squares are
determined by the covering's monodromy.

Every square-tiled surface determines, and is determined by, a pair of
permutations.  Supposing the surface is made of $d$ squares, let
$\Lambda = \{1, 2, ..., d\}$ be the labels of the squares.  We then
consider two permutations, $h$ and $v$, of $\Lambda$ and consider the
surface obtained by gluing the right-hand edge of square $\lambda$ to
the left-hand edge of square $h(\lambda)$, and the top of square
$\lambda$ to the bottom of square $v(\lambda)$.  This surface will be
connected if and only if the group $\langle h, v \rangle$ acts
transitively on $\Lambda$, and we will only consider connected
surfaces.  

\begin{figure}[h!]
  \centering
  \begin{tikzpicture}
    \coordinate (s1) at (0, 2);
    \coordinate (s2) at (0, 1);
    \coordinate (s3) at (1, 1);
    \coordinate (s4) at (2, 1);
    \coordinate (s5) at (1, 0);
    \coordinate (s6) at (2, 0);

    \foreach \x in {1, 2, ..., 6} {
      \draw (s\x) rectangle ($(s\x) + (1, 1)$);
      \node at ($(s\x) + (0.5, 0.5)$) {$\x$};
    }
  \end{tikzpicture}
  \caption{The square-tiled surface on six squares with permutations
    $h = (1)(2 \, 3 \, 4)(5 \, 6)$ and $v = (1 \, 2)(3 \,
    5)(4 \, 6)$.}
  \label{fig:sqtiled}
\end{figure}
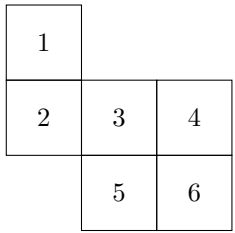

In \cite{GutkinJudge} it was shown that when one translation surface
is a cover of another, $\pi : (X, \omega) \to (Y, \eta)$, their Veech
groups must be commensurable.  (In fact, Gabriela Schmith\"{u}sen has
developed an algorithm for determining the Veech group of a
square-tiled surface \cite{Schmithuesen}.)  Thus $(X, \omega)$ is a
Veech surface if and only if $(Y, \eta)$ is a Veech surface.  Hence
square-tiled surfaces are Veech surfaces since they are covers of tori.

\subsection{Sturmian sequences}
\label{subsec:sturmian}
The flat torus $(\T^2, dz)$ can be thought of as a single unit square
with its top and bottom edges identified, and its left- and right-hand
edges identified as well.  If we label the horizontal edge $V$ and
the vertical edge $H$, then a cutting sequence on the torus is a
biinfinite sequence of $H$'s and $V$'s.

\begin{remark}
  We point out that our convention of having $V$'s on the top and
  bottom of the square and $H$'s on the left- and right-hand sides may
  seem counter-intuitive, but does offer two small advantages over the
  more intuitive choice of putting $V$'s on the vertical edges and
  $H$'s on the horizontal edges.

  We will use this same convention to describe labeled edges on
  square-tiled surfaces, and in that situation crossing an $H$-edge
  corresponds to applying the permutation $h$; and crossing a $V$-edge
  corresponds to applying the permutation $v$, i.e. a side labeled by
  $H$ is shared between two squares glued by a horizontal translation.

  Additionally, with this convention, seeing more $H$'s (respectively,
  $V$'s) in a cutting sequence means the trajectory is more horizontal
  (resp.  vertical) than vertical (horizontal).
\end{remark}

Not all sequences of $H$'s and $V$'s can be obtained as cutting
sequences.  For example, it is easy to see that the sequence will have
isolated $V$'s separated by several $H$'s if the slope of the geodesic
is less than $1$ as in Figure~\ref{fig:torus}.  Having isolated
$V$'s (or $H$'s) is not enough to be a cutting sequence, however.  If
the slope $m$ is less than $1$, for example, then the number of
$H$'s separating $V$'s may be either $\left\lfloor \frac{1}{m} \right\rfloor $ or
$\left\lfloor \frac{1}{m} \right\rfloor +1$ with the number of $V$'s in each \emph{block}
depending on where the geodesic last intersected the horizontal edge
labeled $V$.

\begin{figure}[h!]
  \centering
  \begin{tikzpicture}[scale=3]
    \node[below] at (0.5, 0) {$V$};
    \node[above] at (0.5, 1) {$V$};
    \node[left] at (0, 0.5) {$H$};
    \node[right] at (1, 0.5) {$H$};

    \filldraw[gray] (0.23, 0.45) circle (0.005in);
    \begin{scope}[>=stealth,gray]
      \draw[->] (0.23, 0.45) -- (1, 0.77494);
      \draw[->] (0, 0.77494) -- (0.5333175, 1);
      \draw[->] (0.5333175, 0) -- (1, 0.19694);
      \draw[->] (0, 0.19694) -- (1, 0.61894);
      \draw[->] (0, 0.61894) -- (0.902986, 1);
      \draw[->] (0.902986, 0) -- (1, 0.049399);
      \draw[->] (0, 0.049399) -- (1, 0.4629399);
      \draw[->] (0, 0.4629399) -- (1, 0.8849399);
      \draw[->] (0, 0.8849399) -- (0.27265, 1);
      \draw[->] (0.27265, 0) -- (1, 0.30694);
    \end{scope}

    \draw[thick] (0, 0) rectangle (1, 1);
  \end{tikzpicture}
  \caption{The cutting sequence for the geodesic shown here is $HVHHVHHHVH...$.}
  \label{fig:torus}
\end{figure}
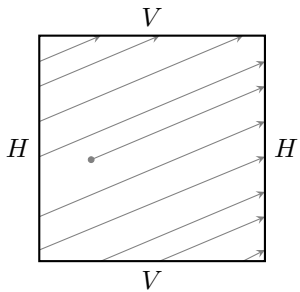

As tori are Veech surfaces, all geodesics in a given direction will be
periodic or uniformly distributed throughout the surface, each case
being determined by the slope of the geodesic.  If the slope is
rational, then the geodesic is periodic.  If the slope is irrational,
then the geodesic will be uniformly distributed.  These two
possibilities are reflected in the types of cutting sequences which
appear.  Periodic geodesics of course give periodic cutting sequences,
while uniformly distributed geodesics give \emph{Sturmian sequences}.

A Sturmian sequence is a sequence on two characters that is not
eventually periodic and whose complexity function, which counts the
number $p(n)$ of distinct subwords of a given length $n$, has the smallest
possible value for a non-eventually periodic sequence: $p(n) = n + 1$.
Sturmian sequences are described in detail in \cite{Arnoux}.

We can determine if a given sequence in the alphabet $E = \{H, V\}$ is
a Sturmian sequence, and so a cutting sequence on the torus, by
performing a combinatorial operation called \emph{derivation}.  As
noted above, these sequences must have one of the letters isolated; 
say $V$'s are isolated for the sake of example.  From each block of
consecutive $H$'s we then delete one $H$ obtaining a new sequence
which is referred to as the \emph{derived sequence} of the original
sequence.  If a given sequence is Sturmian, then its derivation will
also be Sturmian and so again has an isolated letter.  Repeating the
process finitely-many times, the other letter $H$ becomes isolated.
We can then start the process over again, deleting one letter at a
time from blocks of consecutive $V$'s.  

Given a Sturmian sequence $\epsilon \in E^\Z$, we can associate a real
number by the following procedure.  Suppose that the number of $H$'s
appearing between adjacent $V$'s is $a_0$ or $a_0 + 1$ ($a_0$ could be
zero).  Delete $a_0$ $H$'s from each block of consecutive $H$'s
between a pair of $V$'s.  The $H$'s are now the isolated characters.
Now suppose between two adjacent $H$'s there are $a_1$ or $a_1 + 1$
$V$'s.  Again delete blocks of $V$'s of size $a_1$, leaving one $V$ if there were
$a_1 + 1$ $V$'s in the block.  Repeat the process alternating between
deleting $H$'s and $V$'s where at each step there are $a_i$ or
$a_i + 1$ $H$'s (respectively, $V$'s) between the isolated $V$'s
(resp., $H$'s).  Now define the \emph{slope} of the Sturmian sequence
$\epsilon$ to be
\[
  m = m(\epsilon) = a_0 + \cfrac{1}{a_1 + \cfrac{1}{a_2 + \cfrac{1}{a_3 + \ldots}}}
\]
which is slope of the geodesic on the torus with cutting sequence
$\epsilon$.

We will also define the \emph{length} $M(\epsilon)$ of a Sturmian sequence
$\epsilon$ as 
\[
  M = M(\epsilon) = \left\lfloor \frac{1}{m(\epsilon)} \right\rfloor ,
\]
and the \emph{rotation parameter} $\theta(\epsilon)$ of $\epsilon$ as
\[
  \theta = \theta(\epsilon) = \frac{1}{m(\epsilon)} - M(\epsilon).
\]
If we consider a circle which is a horizontal line on the torus, the
first return map of the geodesic flow with slope $m$ is precisely a
rotation by $\theta$.

\subsection{Words and languages}
By an \emph{alphabet} $\mathcal{A}$ we mean some finite set of labels.
A \emph{word} with letters in $\mathcal{A}$ is a finite, ordered list
of elements in $\mathcal{A}$, and the collection of all words with
letters in $\mathcal{A}$ is denoted $\mathcal{A}^*$.  If we define a
binary operation on $\mathcal{A}^*$ by concatenating two words, then
$\mathcal{A}^*$ is the free monoid generated by $\mathcal{A}$.

We will use subscripts to denote individual letters of a word
$w \in \mathcal{A}^*$: $w_0$ is the first character of the word, $w_1$
the second character, and so on.  We will denote the length of a word
by $|w|$.

In general, a \emph{language} with letters in $\mathcal{A}$ is simply
a subset of $\mathcal{A}^*$.  We say that a language is \emph{minimal}
if for each word $w$ in the language there exists an integer $N$ such
that $w$ is a subword of every word in the language of length at least
$N$.  Associated to any language is a dynamical system consisting of
all sequences where every finite subword of a sequence is a word of
the language, together with the shift map.  A \emph{cylinder set}
$[w]$ is the set of all sequences whose first $|w|$ characters give
the word $w$.  These cylinder sets form a base for a topology on the
sequence space making it homeomorphic to a Cantor set and for which
the shift map is continuous.  It is a basic fact that the language is
minimal if and only if the associated topological dynamical system is
minimal: that is, every orbit is dense in the space.  Note that the cylinder
sets form a basis for the topology of this space of sequences, and by
the Kolmogorov extension theorem, a Borel measure on this space is
determined by the measure of each cylinder.

Given a biinfinite sequence of letters in $\mathcal{A}$, say
$u \in \mathcal{A}^\Z$, the \emph{language} of $u$ is the collection
of all finite subwords which appear in $u$.  We will say the sequence
$u$ is \emph{minimal} if its language is minimal.

\subsection{Interval exchange transformations}
An \emph{interval exchange transformation}, abbreviated \emph{IET}, is
a right-continuous bijection from an interval $I$ to itself which is a
piecewise translation having finitely-many discontinuities.  That is,
an IET cuts the interval $I$ up into finitely-many pieces and then
rearranges the pieces.  Such a map is determined by two pieces of
information: the \emph{combinatorial data} of the map is the
permutation describing the rearrangement of the subintervals, and the
\emph{length data} is a vector describing the lengths of the
subintervals.

We will represent the combinatorial data of a given interval exchange
$T : I \to I$ as a pair of bijections.  If $\mathcal{A}$ is a set of
$k$ labels, we consider a pair of bijections,
$\pi_0$ and $\pi_1$, from $\mathcal{A}$ to $\{1, 2, ..., k\}$ so that
$\pi_0^{-1}(1)$ is the label of the left-most subinterval of $I$
before applying $T$; $\pi_1^{-1}(1)$ is the label of the left-most
subinterval of $I$ after applying $T$; $\pi_0^{-1}(2)$ is the label of
the second interval before applying $T$, and $\pi_1^{-1}(2)$ is the
label of the second interval after applying $T$; and so on.

We will use the combinatorial data of an interval exchange to define
two total orderings on the symbols of $\mathcal{A}$.  For each
bijection $\pi: \mathcal{A} \to \{1, 2, ..., d\}$ we define $\leq_\pi$
as follows: for $a, b \in \mathcal{A}$, say $a \leq_\pi b$ if
$\pi(a) \leq \pi(b)$.  By a \emph{$\pi$-interval} of $\mathcal{A}$ we
mean a collection of consecutive elements of $\mathcal{A}$ according
to $\leq_\pi$.

The length data of an interval exchange is given by a vector of
lengths $\ell = (\ell_a)_{a \in \mathcal{A}}$ such that each
$\ell_a > 0$ and $\sum_a \ell_a$ is the length of the interval $I$.
See Figure~\ref{fig:ietexample}.

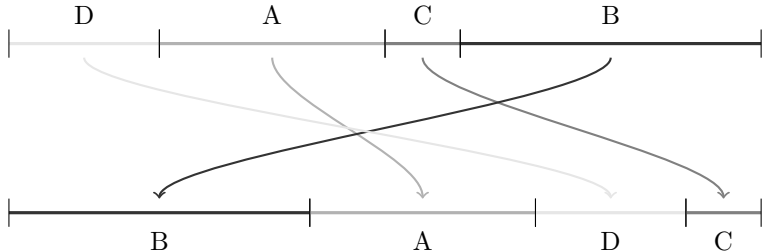
\begin{figure}[h!]
  \centering
  \begin{tikzpicture}[yscale=0.75]
    \def\colorD{black!10!white}
    \def\colorA{black!30!white}
    \def\colorC{black!50!white}
    \def\colorB{black!80!white}

    \draw[very thick, color=\colorD] (0, 0) -- (2.0, 0);
    \draw (0, -0.25) -- (0, 0.25);
    \draw (2.0, -0.25) -- (2.0, 0.25);
    \node at (1.0, 0.5) {D};

    \draw[very thick, color=\colorA] (2.0, 0) -- (5.0, 0);
    \draw (2.0, -0.25) -- (2.0, 0.25);
    \draw (5.0, -0.25) -- (5.0, 0.25);
    \node at (3.5, 0.5) {A};

    \draw[very thick, color=\colorC] (5.0, 0) -- (6.0, 0);
    \draw (5.0, -0.25) -- (5.0, 0.25);
    \draw (6.0, -0.25) -- (6.0, 0.25);
    \node at (5.5, 0.5) {C};

    \draw[very thick, color=\colorB] (6.0, 0) -- (10.0, 0);
    \draw (6.0, -0.25) -- (6.0, 0.25);
    \draw (10.0, -0.25) -- (10.0, 0.25);
    \node at (8.0, 0.5) {B};
    \draw[very thick, color=\colorB] (0, -3) -- (4.0, -3);
    \draw (0, -3.25) -- (0, -2.75);
    \draw (4.0, -3.25) -- (4.0, -2.75);
    \node at (2.0, -3.5) {B};
    \draw[very thick, color=\colorA] (4.0, -3) -- (7.0, -3);
    \draw (4.0, -3.25) -- (4.0, -2.75);
    \draw (7.0, -3.25) -- (7.0, -2.75);
    \node at (5.5, -3.5) {A};
    \draw[very thick, color=\colorD] (7.0, -3) -- (9.0, -3);
    \draw (7.0, -3.25) -- (7.0, -2.75);
    \draw (9.0, -3.25) -- (9.0, -2.75);
    \node at (8.0, -3.5) {D};
    \draw[very thick, color=\colorC] (9.0, -3) -- (10.0, -3);
    \draw (9.0, -3.25) -- (9.0, -2.75);
    \draw (10.0, -3.25) -- (10.0, -2.75);
    \node at (9.5, -3.5) {C};

    \draw[->,thick,color=\colorA] (3.5, -0.25) .. 
      controls (3.5, -1) and (5.5, -2) .. (5.5, -2.75);

    \draw[->,thick,color=\colorC] (5.5, -0.25) .. 
      controls (5.5, -1) and (9.5, -2) .. (9.5, -2.75);

    \draw[->,thick,color=\colorB] (8.0, -0.25) .. 
      controls (8.0, -1) and (2.0, -2) .. (2.0, -2.75);

    \draw[->,thick,color=\colorD] (1.0, -0.25) .. 
      controls (1.0, -1) and (8.0, -2) .. (8.0, -2.75);

  \end{tikzpicture}
  \caption{The interval exchange with combinatorial data
    $\pi_0 = (A, B, C, D) \mapsto (2, 4, 3, 1)$,
    $\pi_1 = (A, B, C, D) \mapsto (2, 1, 4, 3)$ and length data
    $\ell = (0.2, 0.3, 0.1, 0.4)$.}
  \label{fig:ietexample}
\end{figure}

We say that the combinatorial data $(\pi_0, \pi_1)$ is
\emph{irreducible} if there does not exist a $1 \leq j < k$ such that
\[
  \pi_0^{-1}(\{1, 2, ..., j\}) = \pi_1^{-1}(\{1, 2, ..., j\}).
\]
If the combinatorial data of an interval exchange was not
irreducible, then the interval exchange would decompose into two
invariant subintervals for every choice of length data.

Once the combinatorial and length data are known, the actual operation
of the map is easy to describe.  For each $a \in \mathcal{A}$, let
\begin{align*}
  \delta^0_a &= \sum_{b : \pi_0(b) < \pi_0(a)} \ell_b \\
  \delta^1_a &= \sum_{b : \pi_1(b) < \pi_1(a)} \ell_b
\end{align*}
and define the subintervals
\begin{align*}
  I^0_a &= \left[\delta^0_a, \delta^0_{\pi_0^{-1}(\pi_0(a) +
                 1)}\right) \\
  I^1_a &= \left[\delta^1_a, \delta^1_{\pi_1^{-1}(\pi_1(a) +
                 1)}\right).
\end{align*}
Then the $\delta^0_a$ represent the discontinuities of $T$, and
$\delta^1_a$ are the discontinuities of $T^{-1}$.  The map $T$ acts by
translating $I^0_a$ to $I^1_a$:
\begin{align*}
  T(x) &= x - \delta^0_a + \delta^1_a \text{ if } x \in I_a^0, \text{ and } \\
  T^{-1}(x) &= x - \delta^1_a + \delta^0_a \text{ if } x \in I_a^1.
\end{align*}

We say that an interval exchange $T$ satisfies the \emph{infinite
  distinct orbit condition}, often abbreviated \emph{idoc}, if for
each $n > 0$ and each distinct $a, b \in \mathcal{A}$ with
$\pi_0(b) > 1$ we have
\[
  T^n\left(\delta_a^0\right) \neq \delta_b^0.
\]
It was shown in \cite{Keane} that the interval exchange $T$ is minimal
if and only if it satisfies the infinite distinct orbit condition.

Given a finite word $w = w_0 w_1 \cdots w_n \in \mathcal{A}^*$, we
will let $I^0_w$ denote the set of points inside of $I$ which start
off in $I_{w_0}^0$, then proceed to $I_{w_1}^0$ when $T$ is applied,
then $I_{w_2}^0$ when $T^2$ is applied, and so on:
\begin{align*}
  I_w^0 &= \bigcap_{i=0}^n T^{-k}\left(I_{w_i}^0\right).
\end{align*}

Given an interval exchange transformation $T$ and a point $x \in I$,
we can consider the biinfinite sequence of labels obtained by
recording the label $a$ when an iterate $T^n(x)$ is in the interval
$I_a^0$.

If we are given a sequence $u \in \mathcal{A}^\Z$,
it is natural to ask if $u$ represents the symbolic trajectory of some
point under an interval exchange with intervals labeled by
$\mathcal{A}$.  Ferenczi and Zamboni give necessary and sufficient
conditions in \cite{FerencziZamboni} for a sequence of symbols to be
the symbolic trajectory of some interval exchange transformation with
fixed combinatorics.

\begin{theorem}[\cite{FerencziZamboni}]
  \label{thm:FZ}
  The sequence $u \in \mathcal{A}^\Z$ is a symbolic trajectory of an
  interval exchange satisfying Keane's infinite distinct orbit
  condition with irreducible combinatorial data
  $\pi_0, \pi_1 : \mathcal{A} \to \{1, 2, ..., k\}$ if and only if the
  following conditions are satisfied:
  \begin{enumerate}
  \item $u$ is a minimal sequence;
  \item each letter of $\mathcal{A}$ appears in $u$;
  \item for each finite-word $w$ appearing in $u$, the set of letters
    which appear as prefixes of $w$ in $u$ forms a $\pi_1$-interval;
  \item for each finite-word $w$ appearing in $u$, the set of letters
    which appear as suffixes of $w$ in $u$ forms a $\pi_0$-interval;
  \item if $a, b \in \mathcal{A}$ are admissible prefixes of $w$ with
    $a \leq_{\pi_1} b$, and if $y$ is an admissible suffix of $aw$ and
    $z$ an admissible suffix of $bw$, then $y \leq_{\pi_0} z$; and
  \item if $a$ and $b$ are admissible prefixes of $w$, then there is
    exactly one element of $\mathcal{A}$ which is an admissible suffix
    of both $aw$ and $bw$.
  \end{enumerate}
\end{theorem}

This theorem only says that a given sequence $u \in \mathcal{A}^\Z$ is
a symbolic trajectory of \emph{some} interval exchange with the chosen
combinatorial data, but does not tell us if $u$ is a symbolic
trajectory for a particular interval exchange with this combinatorial
data.  Obtaining a particular interval exchange with the given
sequence as a symbolic trajectory is equivalent to finding a
shift-invariant measure for the dynamical system associated with the
sequence.  Given such a measure we can define the
length of the subinterval $I_{a_1 a_2 \cdots a_n}$ as the measure of
the cylinder $[a_1 a_2 \cdots a_n]$.

\section{Interval Exchanges on Square-Tiled Surfaces}
\label{sec:sqiet}
Let $\Lambda = \{1, 2, ..., d\}$ be a set of labels, and let $h$ and
$v$ be two permutations of $\Lambda$ such that $\langle h, v \rangle$
acts transitively on $\Lambda$.  Let $X$ be the corresponding
square-tiled surface.  Denote the geodesic interval at the base of
square $\lambda$ by $I_\lambda$ as in Figure~\ref{fig:sqiet}, and let
$I = \bigcup_\lambda I_\lambda$.  Now consider the first-return map
$T_m : I \to I$ obtained by firing a geodesic ray from $x \in I$ with
slope $m$, and recording where the geodesic first intersects $I$ as in
Figure~\ref{fig:sq}.  In Figure~\ref{fig:sq}, the base of each
$I_\lambda$ is broken into two pieces: a left-hand side and a
right-hand side.  Geodesics in a left-hand side of the base of a
square stay parallel to one another and are followed until reaching
the base of another square.  The paths which geodesics in a left-hand
interval follow are shaded with various textures in
Figure~\ref{fig:sq}.  Similarly, geodesics emitted from the right-hand
side of each base remain parallel and are followed until reaching the
base of another square, and the path followed by these right-hand
trajectories are given different shades of grey.

\begin{figure}[h!]
  \centering
  \hfill
  \begin{subfigure}[b]{0.4\textwidth}
    \begin{tikzpicture}
      \coordinate (s1) at (0, 2);
      \coordinate (s2) at (0, 1);
      \coordinate (s3) at (1, 1);
      \coordinate (s4) at (2, 1);
      \coordinate (s5) at (1, 0);
      \coordinate (s6) at (2, 0);

      \foreach \x in {1, 2, ..., 6} {
        \draw (s\x) -- ($(s\x) + (0, 1)$);
        \draw[thick,dashed] ($(s\x) + (0, 1)$) -- ($(s\x) + (1, 1)$);
        \draw ($(s\x) + (1, 1)$) -- ($(s\x) + (1, 0)$);
      }
      
      \foreach \x in {1, 2, ..., 6} {
        \draw[thick,dashed] (s\x) -- node[font=\small,above] {$I_{\x}$} ($(s\x) + (1, 0)$);
      }
    \end{tikzpicture}
    \caption{The bases of squares form a collection of disjoint
      intervals.}
    \label{fig:sqiet}
  \end{subfigure}
  \hfill
  \begin{subfigure}[b]{0.4\textwidth}
    \begin{tikzpicture}
      \coordinate (s1) at (0, 2);
      \coordinate (s2) at (0, 1);
      \coordinate (s3) at (1, 1);
      \coordinate (s4) at (2, 1);
      \coordinate (s5) at (1, 0);
      \coordinate (s6) at (2, 0);

      \def\LA{gray!75!white}
      \def\LB{gray}
      \def\LD{black!60!white}
      \def\LE{black!80!white}
      \def\LF{black}

      \def\RA{gray!50!white}
      \def\RB{black}
      \def\RC{gray!75!white}
      \def\RD{gray!25!white}
      \def\RE{white}
      \def\RF{gray!66!white}
      
      \draw[pattern=bricks] (0.628, 1) -- (3, 2) -- (2.372, 2) -- (0, 1) -- cycle;
      \filldraw[\RB] (0.628, 1) -- (3, 2) -- (3, 1.844) -- (1, 1)
      -- cycle;
      \draw[pattern=crosshatch] (1.628, 1) -- (3, 1.578984) -- (3, 1.844) --
      (1, 1) -- cycle;
      \filldraw[\RC] (1.628, 1) -- (3, 1.578984) -- (3, 1.422984) -- (2, 1) -- cycle;
      \draw[pattern=north west lines] (2.628, 1) -- (3, 1.157968) -- (3, 1.422984)
      -- (2, 1) -- cycle;
      \filldraw[\RD] (2.628, 1) -- (3, 1.157968) -- (3,
      1) -- cycle;
      \filldraw[\RD] (0, 1) -- (2.372, 2) -- (2, 2) -- (0,
      1.157968) -- cycle;
      \draw[pattern=north west lines] (0, 1.422984) -- (1.372, 2) -- (2, 2) -- (0,
      1.157968);
      \filldraw[\RC] (0, 1.422984) -- (1.372, 2) -- (1, 2) -- (0,
      1.578984) -- cycle;
      \draw[pattern=crosshatch] (0, 1.578984) -- (1, 2) --
      (0.372, 2) -- (0, 1.844) -- cycle;
      \filldraw[\RB] (0, 1.844) -- (0, 2) -- (0.372, 2) -- cycle;
      
      \draw[pattern=vertical lines] (0.628, 2) -- (1, 2.157968) -- (1, 2.422984)  -- (0, 2) -- cycle;
      \draw[pattern=vertical lines] (0, 2.422984) -- (1, 2.844) -- (1, 2.578984) -- (0, 2.157968);
      \draw[pattern=vertical lines] (0, 2.578984) -- (1, 3) --
      (0.372, 3) -- (0, 2.844) -- cycle;
      \filldraw[\RA] (0.628, 2) -- (1, 2.157968) -- (1, 2) -- cycle;
      \filldraw[\RA] (0, 2) -- (1, 2.422984) -- (1, 2.578984) -- (0, 2.157968) -- cycle;
      \filldraw[\RA] (0, 2.422984) -- (1, 2.844) -- (1, 3) -- (0, 2.578984) -- cycle;
      \filldraw[\RA] (0, 2.844) -- (0, 3) -- (0.372, 3) -- cycle;

      \draw[pattern=dots] (1.628, 0) -- (3, 0.578984) -- (3, 0.844) --
      (1, 0) -- cycle;
      \draw[pattern=dots] (1, 0.578984) -- (2, 1) --
      (1.372, 1) -- (1, 0.844) -- cycle;
      \draw[pattern=crosshatch dots] (2.628, 0) -- (3, 0.157968) -- (3, 0.422984)
      -- (2, 0) -- cycle;
      \draw[pattern=crosshatch dots] (1, 0.422984) -- (2.372, 1) -- (3, 1) -- (1,
      0.157968);
      \filldraw[\RE] (1, 0.422984) -- (2.372, 1) -- (2, 1) -- (1,
      0.578984) -- cycle;
      \filldraw[\RE] (1.628, 0) -- (3, 0.578984) -- (3, 0.422984) -- (2, 0) -- cycle;
      \filldraw[\RF] (2.628, 0) -- (3, 0.157968) -- (3,
      0) -- cycle;
      \filldraw[\RF] (1, 0) -- (3, 0.844) -- (3, 1) -- (1, 0.157968) -- cycle;
      \filldraw[\RF] (1, 0.844) -- (1, 1) -- (1.372, 1) -- cycle;
      
      \foreach \x in {1, 2, ..., 6} {
        \draw (s\x) rectangle ($(s\x) + (1, 1)$);
      }
    \end{tikzpicture}
    \caption{We consider the first-return map on this collection of intervals.}
    \label{fig:sq}
  \end{subfigure}
  \caption{We consider the interval exchanges which are given by the
    first-return map on the bases of the squares.}
  \label{fig:firstreturn}
\end{figure}
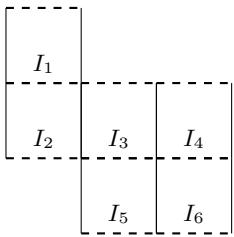
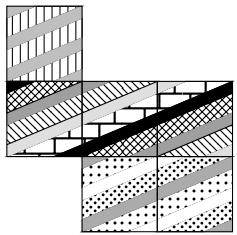

In the following we assume that the slope of a geodesic is $m > 0$ and
the geodesic flow is to the North-East on the surface.

\begin{lemma}
  \label{lemma:skewproduct}
  The first-return map on the base of the squares in a square-tiled
  surface has the form
  \[
    T_m(x, \lambda) = \begin{cases}
      (R_\theta(x), vh^M(\lambda)) &\text{ if } x \in [0, 1 - \theta) \\
      (R_\theta(x), vh^{M+1}(\lambda)) &\text{ if } x \in [1 - \theta, 1)
    \end{cases}
  \]
  where $\theta \in [0, 1)$ and $M \in \{0, 1, 2, ..., D-1\}$ are
  values depending on $m$, $D$ depends on the surface, and $R_\theta$
  is the rotation of $[0, 1)$ by $\theta$.
\end{lemma}
\begin{proof}
  Since each $I_\lambda$ is a unit interval, we can think of $I$ as
  $[0, 1) \times \Lambda$, identifying $I_\lambda$ with $[0, 1) \times
  \{\lambda\}$.  Suppose that the horizontal permutation
  $h$ consists of cycles $h_1$, $h_2$, ..., $h_k$ and denote the length of
  the cycle $h_i$ by $|h_i|$.  Define
  \begin{align*}
    D &:= \lcm(|h_1|, |h_2|, ..., |h_k|),   \\
    M &:= \left\lfloor \frac{1}{m} \right\rfloor \pmod{D}, \\
    \theta &:= \frac{1}{m} - \left\lfloor \frac{1}{m} \right\rfloor,
             \text{ and } \\
    R_\theta(x) &:= x + \theta - \lfloor x + \theta \rfloor.
  \end{align*}
  If $x \in I_\lambda$, then the geodesic ray to the
  North-East emitted from $x$ with slope $m$ cuts across the vertical
  sides of squares in the same horizontal cylinder as square $\lambda$
  before eventually reaching the base of some square $\lambda'$ in one
  of the horizontal cylinders connected to the cylinder containing
  $\lambda$.  Because the squares are $1 \times 1$ and the geodesic
  has slope $m$, the horizontal distance traveled by the geodesic
  before reaching $I_{\lambda'}$ is $\frac{1}{m}$.

  To make the arithmetic simpler, we develop the cylinder containing
  square $\lambda$ in the plane, choosing coordinates so that the
  left-most point of $I_\lambda$ is $(0, 0)$.  See Figure~\ref{fig:sqfirstreturn}.

  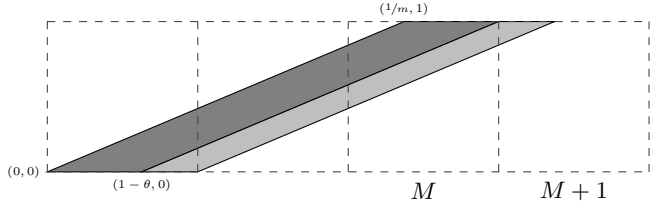
\begin{figure}[h!]
    \centering
    \begin{tikzpicture}[scale=2,font=\tiny]
      \filldraw[fill=gray]
      (0, 0) node[left,scale=0.75] {$(0, 0)$}
      -- (2.372, 1) node[above,scale=0.75] {$(\sfrac{1}{m}, 1)$} 
      -- (3, 1)
      -- (0.628, 0) node[below,scale=0.75] {$(1-\theta, 0)$}
      -- cycle;

      \filldraw[fill=gray!50!white]
      (0.628, 0)
      -- (3, 1)
      -- (3.372, 1)
      -- (1, 0)
      -- cycle;

      \draw[step=1cm,dashed,gray!50!black] (0, 0) grid (4, 1);

      \node[font=\small] at (2.5, -0.125) {$M$};
      \node[font=\small] at (3.5, -0.125) {$M+1$};
    \end{tikzpicture}
    \caption{Computing the first-return map.}
    \label{fig:sqfirstreturn}
  \end{figure}
  
  The geodesic emitted from $x$ leaves the
  cylinder at the point $(y, 1)$ where
  \[
    y := x + \frac{1}{m} = x + M + \theta.
  \]
  When $x < 1 - \theta$, $y$ is in $[M, M+1)$, and so cuts across $M$
  squares before leaving the cylinder and moving up into square
  $vh^M(\lambda)$.  When $x \geq 1 - \theta$,
  $y \in [M+1, M+2)$ and so the geodesic leaves the
  cylinder by moving into square $vh^{M+1}(\lambda)$.

  As we identify the base of square $\lambda'$ with the interval $[0,
  1)$, the $x$-coordinate of our geodesic upon entering $I_\lambda'$ is
  \[
    x + \frac{1}{m} - \left\lfloor x + \frac{1}{m} \right\rfloor = R_\theta(x).
  \]
  
  The cycles of $h$ determine horizontal cylinders on the surface where
  the length of a cylinder equals the length of the corresponding cycle.
  Given a cycle of length $k$, then for each square $\lambda$ in this
  cylinder, $h^M(\lambda) = h^{M + nk}(\lambda)$ for each $n \in \Z$.
  Hence we only care about the value of $M$ modulo $D$.  
\end{proof}
  
\begin{lemma}
  \label{lemma:conjugate}
  The first-return map $T_m$ described in
  Lemma~\ref{lemma:skewproduct} is conjugate to an interval exchange
  on $[0, d)$.
\end{lemma}
\begin{proof}
  Keeping $M$ and $\theta$ as defined in Lemma~\ref{lemma:skewproduct},
  let $T$ be the interval exchange on $2d$ subintervals of $[0, d)$
  with labels from $\mathcal{A} = \Lambda \times \{L, R\}$,
  combinatorial data
  \begin{align*}
    \pi_0(\lambda, s) =& \begin{cases}
      2\lambda - 1 & \text{ if } s = L \\
      2\lambda & \text{ if } s = R \\
    \end{cases} \\
    & \\
    \pi_1(\lambda, s) =& \begin{cases}
      2vh^M(\lambda) & \text{ if } s = L \\
      2vh^{M+1}(\lambda) - 1 & \text{ if } s = R \\
    \end{cases}
  \end{align*}
  and length data
  \[
    \ell_{\lambda, s} = \begin{cases}
      1 - \theta & \text{ if } s = L \\
      \theta & \text{ if } s = R
    \end{cases}.
  \]
  With this choice of combinatorial and length data, the
  discontinuities of $T$ and $T^{-1}$ are respectively
  \[
    \delta^0_{(\lambda, s)}
    = \begin{cases}
      \lambda - 1 & \text{ if } s = L \\
      \lambda - \theta & \text{ if } s = R
    \end{cases} 
    \quad \text{ and } \quad
    \delta^1_{(\lambda, s)} = \begin{cases}
      vh^M(\lambda) - 1 + \theta& \text{ if } s = L \\
      vh^{M+1}(\lambda) - 1 & \text{ if } s = R.
    \end{cases}
  \]
  
  Now consider the map $\phi : [0, 1) \times \Lambda \to [0, d)$
  defined by
  \[
    \phi(x, \lambda) = x + \lambda - 1
  \]
  It is easy to verify that $\phi T_m = T \phi$.
  If $x \in [0, 1 - \theta)$, then
  \begin{align*}
    \phi T_m(x, \lambda)
    &= \phi(R_\theta(x), vh^M(\lambda)) \\
    &= R_\theta(x) + vh^M(\lambda) - 1 \\
    &= x + \theta + vh^M(\lambda) - 1 \\
    &= x + \lambda - 1 - (\lambda - 1) + vh^M(\lambda) -1 + \theta \\
    &= x + \lambda - 1 - \delta_{(\lambda,L)}^0 + \delta_{(\lambda,L)}^1 \\
    &= T(x + \lambda - 1) \\
    &= T\phi(x, \lambda).
  \end{align*}
  The proof is similar if $x \in [1 - \theta, 1)$.
\end{proof}

In the above we assumed the slope $m$ was positive and the flow was to
the North-East just to simplify the statements and proofs of
Lemma~\ref{lemma:skewproduct} and Lemma~\ref{lemma:conjugate}.  If we
considered flowing in another direction, then the above arguments
still hold except $h$ and/or $v$ may be replaced by their inverses,
and $R_\theta$ may be replaced by $R_{-\theta}$.  For example,
flowing to the North-West with slope $m < 0$ on the surface given by
$(h, v)$ is equivalent to flowing to the North-East with slope $-m$ on
the surface determined by $(h^{-1}, v)$, and similarly for the other
directions.

\section{Consistent Sequences and Symmetric Sturmian Sequences}
\label{sec:consistency}
\subsection{Consistent sequences}
Let $\lambda$ be the label of some square on a square-tiled surface,
and consider the geodesic intervals corresponding to the left-hand and
bottom edges of $\lambda$, which we will label $(\lambda, H)$ and
$(\lambda, V)$, respectively, as in Figure~\ref{fig:labels}.  Now
consider the geodesic flow with slope $m$ on the surface, where for
simplicity we again consider the case that $m > 0$ and the flow is to
the North-East.  When the flow enters square $\lambda$ it must cross
either $(\lambda, H)$ or $(\lambda, V)$, and then must leave the
square by crossing either $(v(\lambda), V)$ or $(h(\lambda), H)$.  The
order in which the geodesic crosses these edges gives a sequence in
$(\Lambda \times E)^\Z$.  We will adopt the convention that if
the geodesic hits the corner of a square which is not a conical
singularity, then two symbols are recorded: either $(h(\lambda), H) \,
(vh(\lambda), V)$ or $(v(\lambda), V) \, (hv(\lambda), H)$.

If the upper right-hand corner of square $\lambda$ is not a conical
singularity, then $vh(\lambda) = hv(\lambda)$.  The squares $\lambda$ satisfying
this property are called \emph{good squares}, and squares with a
singularity in the upper right-hand corner are \emph{bad squares}.

\begin{figure}[h!]
  \centering
  \begin{tikzpicture}[scale=1.25,transform shape]
    \coordinate (s1) at (0, 2);
    \coordinate (s2) at (0, 1);
    \coordinate (s3) at (1, 1);
    \coordinate (s4) at (2, 1);
    \coordinate (s5) at (1, 0);
    \coordinate (s6) at (2, 0);

    \foreach \x in {1, 2, ..., 6} {
      \draw (s\x) rectangle ($(s\x) + (1, 1)$);
      \node[above,scale=0.4] at ($(s\x) + (0.5, 0)$) {$(\x, V)$};
      \node[left,rotate=90,scale=0.4] at ($(s\x) + (0.15, 0.75)$) {$(\x, H)$};
    }
  \end{tikzpicture}
  \caption{The curves we consider in the cutting sequence.}
  \label{fig:labels}
\end{figure}
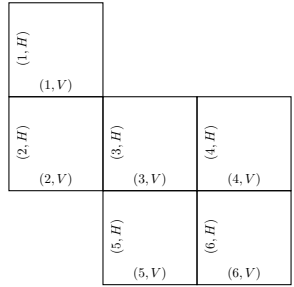

Our basic question is which sequences in
$(\Lambda \times E)^\Z$ correspond to cutting sequences on the
surface.  It is obvious that not all sequences of symbols will be
cutting sequences.  The simplest property all cutting sequences
must satisfy is what we call \emph{consistency}.  We say that a
sequence $(\lambda_n, \epsilon_n) \in (\Lambda \times E)^\Z$ is
\emph{consistent} if $\lambda_{n+1} = h(\lambda_n)$ when $\epsilon_{n+1} = H$,
and $\lambda_{n+1} = v(\lambda_n)$ when $\epsilon_{n+1} = V$.  To simplify
notation we will let $E$ act on $\Lambda$ by
$H\cdot\lambda = h(\lambda)$ and $V\cdot\lambda = v(\lambda)$ with
$h, v \in \mathfrak{S}_{\Lambda}$ fixed permutations.

For example, consider the L-shaped surface built from three squares
with permutations $h = (1)(2 \, 3)$ and $v = (1 \, 2)(3)$ as in
Figure~\ref{fig:unliftable}.  A sequence containing
\[
  ... (1, H) \, (2, V) \, (3, H) \, (2, H) ...
\]
may be consistent (whether it is in fact consistent or not depends of
course on the rest of the sequence, but nothing shown here violates
consistency), while a sequence containing
\[
  ... (1, H) \, (2, V) \, (1, H) \, (1, H) ...
\]
\emph{can not} be consistent because of the $(1, H)$ which follows
$(2, V)$.  If the pair $(\lambda, H)$ follows $(2, V)$ then we
\emph{must} have $\lambda = 3$ as $h(2) = 3$.

Since a square-tiled surface $X$ is a cover of the square torus
$\T^2$, geodesics on $X$ project to geodesics on $\T^2$.  Similarly,
cutting sequences on $X$ project to cutting sequences on $\T^2$.  In
particular, if $(\lambda_n, \epsilon_n)$ is a cutting sequence on the
square-tiled surface $X$, then $\epsilon$ should be a cutting sequence
on $\T^2$.  We can try to go in the other direction, lifting a cutting
sequence on $\T^2$ to a cutting sequence on $X$.  In particular, given
an initial square $\lambda_0$ and a cutting sequence $\epsilon$ on
$\T^2$, we can construct a \emph{combinatorial lift} by defining
$\lambda_{n+1} = \epsilon_{n+1} \cdot \lambda_n$.  We may choose
$\lambda_0$ to be any initial square, and so on a square-tiled surface
built from $d$ squares there are $d$ different combinatorial lifts of
each cutting sequence on the torus.  Combinatorial lifts are
consistent, but not all consistent sequences are combinatorial lifts.
Additionally, every cutting sequence on $X$ is a combinatorial lift of
some cutting sequence on $\T^2$, but the converse is not true.

\subsection{A combinatorial lift which is not a cutting sequence}
\label{subsec:counterexample}
Consider the cutting sequence associated to the geodesic on the square
torus with slope $m = \frac{211}{500} = 0.422$ which passes through
the point $(\frac{133}{211}, 0)$, where $(0, 0)$ is the lower
left-hand corner of the square.  This geodesic crosses the vertical
edge $H$ twice before intersecting the top right-hand corner of the
square, but then continues from the bottom left-hand corner giving the
cutting sequence $\epsilon = ...VHHHV...$ as in
Figure~\ref{fig:cornertorus}.  Here we made the choice that hitting
the corner should correspond to recording $HV$.  We could have just as
easily decided to instead record $VH$ to obtain the sequence
$\epsilon' = ...VHHVH...$.  We will assume that $\epsilon_0$ and
$\epsilon'_0$ are the first characters that appear in the strings
above; that is, $\epsilon_0 = V$, $\epsilon_1 = H$, $\epsilon_2 = H$,
$\epsilon_3 = H$, $\epsilon_4 = V$, and $\epsilon_0' = V$, $\epsilon_1' = H$,
$\epsilon_2' = H$, $\epsilon_3' = V$, $\epsilon_4' = H$, and so on.


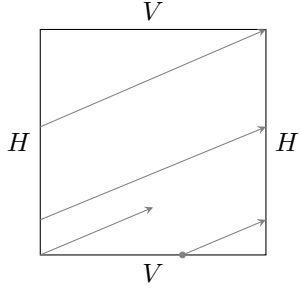
\begin{figure}[h!]
  \centering
  \begin{tikzpicture}[scale=3]
    \draw (0, 0) rectangle (1, 1);
    \node[below] at (0.5, 0) {$V$};
    \node[above] at (0.5, 1) {$V$};
    \node[left] at (0, 0.5) {$H$};
    \node[right] at (1, 0.5) {$H$};

    m = 0.422, rotation is 0.36967
    \filldraw[\trajColor] (0.63033, 0) circle (0.005in);
    \begin{scope}[>=stealth, \trajColor]
      \draw[->] (0.63033, 0) -- (1, 0.156);
      \draw[->] (0, 0.156) -- (1, 0.568);
      \draw[->] (0, 0.568) -- (1, 1);
      \draw[->] (0, 0) -- (0.5, 0.211);
    \end{scope}
  \end{tikzpicture}
  \caption{On the torus, geodesics are still defined even when they hit
    the corner.}
  \label{fig:cornertorus}
\end{figure}

We will now build a consistent sequence which is not the cutting
sequence of a geodesic on a square-tiled surface.  We consider the
L-shaped surface built from three squares which has permutations
$h = (1)(2 \, 3)$ and $v = (1 \, 2)(3)$ shown in
Figure~\ref{fig:unliftable}.  Now consider the sequences
$(\lambda_n, \epsilon_n)$ and $(\lambda_n', \epsilon_n')$ on
$\left(\{1, 2, 3\} \times \{H, V\}\right)^\Z$ which are defined by
$\lambda_0 = \lambda_0' = 2$ and
$\lambda_{n+1} = \epsilon_{n+1} \cdot \lambda_n$.  The sequence
$\lambda'$ is constructed similarly, using $\epsilon'$ in place of
$\epsilon$.  We then have
\begin{align*}
  (\lambda_n, \epsilon_n) &= ... (2, V) (3, H) (2, H) (3, H) (3, V) ... \\
  (\lambda_n', \epsilon_n') &= ... (2, V) (3, H) (2, H) (1, V) (1, H) ...
\end{align*}
By construction both of these sequences are consistent with the
permutations defining the surface and are combinatorial lifts of the
Sturmian sequences $\epsilon$ and $\epsilon'$, but neither sequence
is a cutting sequence on this surface.  If
$(\lambda_n, \epsilon_n)$ was the cutting sequence of $\gamma$ and
$(\lambda_n', \epsilon_n')$ the cutting sequence of $\gamma'$, then
$\gamma$ and $\gamma'$ would project to the geodesic in
Figure~\ref{fig:cornertorus}.  However, no infinitely-long geodesic on
this L-shaped surface can project to this geodesic because the
geodesic on the torus passes through a corner which is one of the
branch points of the cover: on the L-shaped surface this is a conical
singularity and the geodesic flow ends upon hitting the singularity.

\begin{figure}[h!]
  \centering
  \begin{tikzpicture}
    \coordinate (s1) at (0, 1);
    \coordinate (s2) at (0, 0);
    \coordinate (s3) at (1, 0);

    \foreach \x in {1, 2, 3} {
      \draw (s\x) rectangle ($(s\x) + (1, 1)$);
      \draw[\trajColor,>=stealth,->] 
        (s\x) -- ($(s\x) + (0.5, 0.211)$) node[font=\tiny,right] {?};
    }

    \draw[\trajColor,->] (0, 0.568) -- (1, 1);
    \draw[\trajColor,->] (0.63033, 0) -- (2, 0.568);
    \draw[\trajColor,->] (0, 0.568) -- (1, 1);

    \filldraw[\trajColor] (0.707107, 0) circle (0.005in);
  \end{tikzpicture}
  \caption{We can easily produce sequences of labels which are
    consistent with the gluings of a square-tiled surface, but which
    do not represent cutting sequences.}
  \label{fig:unliftable}
\end{figure}
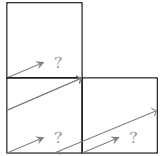

The choice of $\lambda_0 = 2$ is arbitrary in this example.  The same
phenomenon of constructing a cutting sequence which describes a
geodesic which intersects a cone point would happen for any choice of
$\lambda_0$ on the L-shaped surface since the corner of each square is
a cone point.  On other surfaces where some corners are cone points
and some are not, however, the choice of initial square $\lambda_0$
matters.  We will see later that cutting sequences which describe
geodesics passing through a corner must satisfy a certain symmetry
condition, such a sequence is called \emph{almost symmetric} below.
These almost symmetric sequences are candidates for cutting sequences
of the torus which do not lift to the square-tiled surface, but some
may lift and some may not, depending on the choice of $\lambda_0$.  We
will encode the information of which sequences lift and which do not
by considering walks on a special type of graph and will see that
there is a way of determining precisely which almost symmetric
sequences describe geodesic which intersect a cone point according to
properties of this walk.

Every geodesic on the square torus $\mathbb{T}^2$ which does not pass
through the corner point has $d$, the number of squares, lifts to a
square-tiled surface by basic covering theory.  Thus every cutting
sequence of a geodesic on $\mathbb{T}^2$, which does not describe a
geodesic passing through the corner, has $d$ combinatorial lifts to
the square-tiled surface.  Hence the only serious obstacle to
describing cutting sequences on a square-tiled surface is classifying
those sequences which describe geodesics passing through the corner
point. 

\subsection{Symmetric Sturmian sequences}
We will call a Sturmian sequence $\epsilon$ \emph{odd symmetric} if there
exists an $N$ such that $\epsilon_{N+k} = \epsilon_{N-k}$, and \emph{even
  symmetric} if there exists an $N$ such that
$\epsilon_{N+k} = \epsilon_{N-k-1}$ for every $k \geq 0$. We will say $\epsilon$
is \emph{almost symmetric} if there exists an $N$ such that
$\epsilon_{N+k} = \epsilon_{N-k-1}$ for every $k \geq 1$ and
$\epsilon_N \neq \epsilon_{N-1}$.  Below is an example of an odd symmetric
sequence, an even symmetric sequence, and two almost symmetric
sequences.  The points around which the sequences are symmetric are
highlighted. 
\begin{center}
  \begin{tabular}{cc}
  $... HVVHVVV\mathbf{H}VVVHVV...$ & Odd symmetric\\
  $... VVHVVH\mathbf{VV}HVVHVV...$ & Even symmetric\\
  $... HVVHVV\mathbf{VH}VVHVVV...$ & Almost symmetric\\
  $... HVVHVV\mathbf{HV}VVHVVV...$ & Almost symmetric\\
  \end{tabular}
\end{center}

By a \emph{symmetric} Sturmian sequence we
will mean a sequence which is odd symmetric, even symmetric, or almost
symmetric.

Recall that the torus is a hyperelliptic surface and the hyperelliptic
involution $\iota : \T^2 \to \T^2$ acts as a $180^\circ$-rotation of
the square about the center point.  There are four Weierstrass
  points fixed by $\iota$ indicated in Figure~\ref{fig:markedpts}.

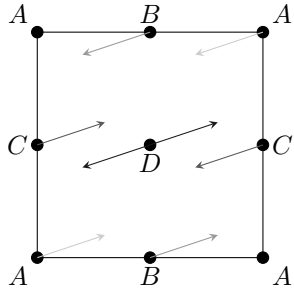
\begin{figure}[h!]
  \centering
  \begin{tikzpicture}[scale=3]
    \draw (0, 0) rectangle (1, 1);
    \filldraw (0, 0) circle (0.01in) node[below left] {$A$};
    \filldraw (1, 0) circle (0.01in) node[below right] {$A$};
    \filldraw (0, 1) circle (0.01in) node[above left] {$A$};
    \filldraw (1, 1) circle (0.01in) node[above right] {$A$};
    \filldraw (0.5, 0) circle (0.01in) node[below] {$B$};
    \filldraw (0.5, 1) circle (0.01in) node[above] {$B$};
    \filldraw (0, 0.5) circle (0.01in) node[left] {$C$};
    \filldraw (1, 0.5) circle (0.01in) node[right] {$C$};
    \filldraw (0.5, 0.5) circle (0.01in) node[below] {$D$};

    \draw[->,>=stealth,black!20!white] (0, 0) -- (0.3, 0.1);
    \draw[->,>=stealth,black!20!white] (1, 1) -- (0.7, 0.9);

    \draw[->,>=stealth,black!40!white] (0.5, 0) -- (0.8, 0.1);
    \draw[->,>=stealth,black!40!white] (0.5, 1) -- (0.2, 0.9);

    \draw[->,>=stealth,black!65!white] (0, 0.5) -- (0.3, 0.6);
    \draw[->,>=stealth,black!65!white] (1, 0.5) -- (0.7, 0.4);

    \draw[->,>=stealth,black!90!white] (0.5, 0.5) -- (0.8, 0.6);
    \draw[->,>=stealth,black!90!white] (0.5, 0.5) -- (0.2, 0.4);
  \end{tikzpicture}
  \caption{The four Weierstrass points on the torus together with 
    geodesics through those points.}
  \label{fig:markedpts}
\end{figure}

\begin{lemma}
  \label{lemma:symmetry}
  A Sturmian sequence is symmetric if and only if it represents the
  cutting sequence of a uniformly distributed geodesic which passes
  through a Weierstrass point of the torus.  Furthermore,
  \begin{enumerate}[(i)]
  \item even symmetric Sturmian sequences correspond to
    geodesics passing through the center point $D$;
  \item almost symmetric sequences pass through the corner point $A$.
  \item odd symmetric
    sequences correspond to geodesics passing through $B$ or $C$; and 
  \end{enumerate}
\end{lemma}

Recall that we are considering uniformly distributed geodesics which
must have irrational slope, and such a geodesic can not pass through two
distinct Weierstrass points.

\begin{proof}[Proof of Lemma~\ref{lemma:symmetry}]
  Let $\gamma$ denote the geodesic whose cutting sequence is the
  Sturmian sequence $\epsilon$.  Suppose that $\gamma$ passes through
  the point $P$ on the side of the square labeled $\epsilon_0$.
  Applying $\iota$ thus gives a geodesic $\iota(\gamma)$ which passes
  through the point $\iota(P)$ which must also be on the side of the
  square labeled $\epsilon_0$, since $\iota$ preserves the sides of
  the square.  Furthermore, the cutting sequence $\iota(\gamma)$ is
  the same as the cutting sequence of $\gamma$, provided we walk along
  the geodesics $\gamma$ and $\iota(\gamma)$ according to how those
  geodesics are oriented.  If we had adopted the convention to always
  walk along the geodesic in the positive (North-East) direction, or
  always in the negative (South-West) direction, then the cutting
  sequences would reverse.  When $\gamma$ passes through a Weierstrass
  point, however, the geodesic $\gamma$ is preserved and we see the
  same cutting sequences, emitted from that Weierstrass point, in both
  the positive and negative direction, giving the symmetries as
  described above.  The only oddity here is to recall that when the
  Weierstrass point under consideration is the corner point we record
  the symbols $VH$, or $HV$, and this results in the ``almost
  symmetric'' sequence.  This establishes one direction of the lemma,
  and for the converse we consider the three cases of even symmetric,
  odd symmetric, and almost symmetric separately.

  In each case we have the following setup: supposing $\epsilon$ is
  symmetric, there exist points $P$ and $Q$ on
  $\gamma \cap (H \cup V)$ such that the geodesic emitted from $P$ in
  the positive direction has the same cutting sequence as the geodesic
  emitted from $Q$ in the negative direction.  Equivalently, the
  geodesic emitted from $\iota(Q)$ in the positive direction gives the
  same cutting sequence as the geodesic emitted from $P$ in the
  positive direction.

  

  \begin{enumerate}[(i)]
  \item Suppose $\epsilon$ is even symmetric, so the ray emitted from
    $P$ in the positive direction is
    $\epsilon' = \epsilon_0 \epsilon_1 \epsilon_2 \cdots$, and the ray
    emitted from $Q$ in the negative direction is
    $\epsilon' = \epsilon_{-1} \epsilon_{-2} \epsilon_{-3} \cdots$
    with $\epsilon_{-1} = \epsilon_0$, $\epsilon_{-2} = \epsilon_1$ and
    so on.  As the ray emitted from $P$ gives $\epsilon_{\geq 0}$
    while the ray emitted from $Q$ gives $\epsilon_{< 0}$, the
    geodesic $\gamma$ connects $Q$ to $P$ without first intersecting
    $H \cup V$.  Additionally, since $\epsilon_0 = \epsilon_{-1}$, $P$
    and $Q$ must both be on the same edge, $H$ or $V$.  Since $P$ and
    $Q$ are involutes, however, the only geodesic that connects them
    without first passing through the sides $H$ or $V$ must pass
    through the center point $D$ of the square.
  \item Suppose $\epsilon$ is almost symmetric, with the ray emitted from
    $P$ being $\epsilon_{\geq 1}$ and the ray emitted from $Q$ being
    $\epsilon_{< -1}$ with $\epsilon_{i} = \epsilon_{-i-1}$ for $i > 0$,
    and $\epsilon_{-1}\epsilon_0$ being one of $HV$ or $VH$.

    The portion of the geodesic from $Q$ to $P$ is responsible for the
    $HV$ (or $VH$) at $\epsilon_{-1}\epsilon_0$.  Suppose we followed
    $\gamma$ backwards from $P$ to some point $R$ on the vertical edge
    $H$.  We would likewise follow $\gamma$ forwards from $Q$ to
    $\iota(R)$ which would also be on $H$.  If $R \neq \iota(R)$, then
    this would produce two distinct $H$'s, which do not appear at
    positions $\epsilon_{-1}$ and $\epsilon_0$ in our sequence.
    Similarly, if $R$ were on the horizontal edge $V$, then $\iota(R)$
    would be as well, and if $R \neq \iota(R)$ we would have two
    distinct $V$'s which do not appear at positions $\epsilon_{-1}$
    and $\epsilon_0$ in the sequence.  We must then have that
    $R = \iota(R)$.  If $R$ was either of the Weierstrass points
    labeled $B$ or $C$ in Figure~\ref{fig:markedpts} we would thus
    produce a single character in our sequence, but not the two
    characters we lack.

    Since our sequence contains $VH$ (or $HV$) and intersecting a
    non-Weierstrass point would produce $HH$ or $VV$, and a
    Weierstrass point which is not $A$ would produce a since $H$ or
    single $V$, the only remaining option is for the geodesic to pass
    through the Weierstrass point labeled $A$ at the corner of the
    square to produce the $VH$ (or $HV$).
  \item If $\epsilon$ is odd symmetric, then we may choose $P$ and $Q$
    such that the ray emitted from $P$ gives $\epsilon_{> 0}$ while
    the ray emitted from $Q$ gives $\epsilon_{< 0}$ with $\epsilon_{i}
    = \epsilon_{-i}$ for $i > 0$.  This again implies that $P$ and $Q$
    are both on $H$ or both on $V$.

    Following $\gamma$ from $Q$ to $P$ crosses the edge $\epsilon_0$,
    while following $\gamma$ backwards from $P$ to $Q$ crosses
    $\epsilon_0$ at the same point.  Since $Q$ and $P$ are involutes,
    this point on $\epsilon_0$ is fixed, and so is either $A$, $B$, or
    $C$.  The point can not be $A$, however, as geodesics through $A$
    give almost symmetric sequences, and our sequence is not almost
    symmetric.  Thus an odd symmetric sequence must pass through
    either $B$ or $C$.
  \end{enumerate}
\end{proof}

Recall that given a Sturmian sequence $\epsilon$, its
\emph{derived sequence} is the sequence $\epsilon'$ obtained from $\epsilon$
by deleting one character from each block of the non-isolated
character.  Geometrically, this derivation corresponds to applying the
vertical, downward Dehn twist when $V$ is isolated, or the horizontal,
left-ward Dehn twist when $H$ is isolated.  This alters the direction
of a geodesic, and the derived sequence is simply the cutting sequence of
this new geodesic.

\begin{lemma}
  \label{lemma:almostsymmetricderiv}
  If a Sturmian sequence $\epsilon$ is almost symmetric, then so is its
  derivation $\epsilon'$.
\end{lemma}
\begin{proof}
  If $\epsilon$ is almost symmetric, then it is the cutting sequence of
  a geodesic through the corner of the square.  Since the corners are
  preserved by the Dehn twists above, the derivation
  $\epsilon'$ also corresponds to a geodesic through the corner and so
  is almost symmetric as well.
\end{proof}

\section{Characterizing Cutting Sequences}
\label{sec:proof}
In this section we prove a characterization theorem for cutting
sequences on square-tiled surfaces.  We will first make note of
certain necessary conditions such a cutting sequence must satisfy, and
then convert a sequence satisfying those conditions into a symbolic
trajectory for an interval exchange and verify the Ferenczi-Zamboni
conditions.  We then show that there exists a shift-invariant measure
agreeing with the lengths of intervals on our square-tiled surface to
conclude that the sequence is a symbolic trajectory for the interval
exchange on our surface.

\subsection{Necessary conditions}
Recall that a sequence $(\lambda_n, \epsilon_n)$ in letters
$\Lambda \times E$ is \emph{consistent} if
$\lambda_{n+1} = \epsilon_{n+1} \cdot \lambda_n$ and is a
\emph{combinatorial lift} if $\epsilon_n$ is the cutting sequence of a
geodesic on $\T^2$.  Obviously being a combinatorial lift is a
necessary condition for $(\lambda_n, \epsilon_n)$ to be a cutting
sequence, but as the example in Section~\ref{subsec:counterexample}
shows it is in general not a sufficient condition.

We will say that $(\lambda_n, \epsilon_n)$ is \emph{almost symmetric
  around a bad square} if $\epsilon_n$ is an almost symmetric sequence,
so $\epsilon_{N+k} = \epsilon_{N-k-1}$ for some $N$ and all $k \geq 1$ but
$\epsilon_N \neq \epsilon_{N-1}$, and $\lambda_N$ is a bad square.

\subsection{The periodic case}
If $(\lambda_n, \epsilon_n)$ is a periodic combinatorial lift, then
$\epsilon_n$ must be periodic.  Periodic cutting sequences on the torus
correspond to geodesics of rational slope on the torus.  Any lift of
such a geodesic to the square-tiled surface with permutations $h$ and
$v$ is a periodic geodesic as well by basic lifting properties of
covering spaces and the fact that the square-tiled surface and torus
are locally isometric.

\subsection{The non-symmetric case}
Combinatorial lifts of Sturmian sequences which are not almost
symmetric, and so do not represent cutting sequences passing through a
corner of the square, easily lift to cutting sequences on every
square-tiled surface.

It is clear, however, that there may be geodesics on a square-tiled
surface whose projection to the torus gives an almost symmetric
cutting sequence.  This can happen if the geodesic goes through the
corner of a good square.  To determine when this happens we will
associate a family of graphs to the surface which describe possible
transitions between the squares when the geodesic leaves one
horizontal cylinder for another.

\subsection{The $\Gamma^M$ graphs}
In the description of $T_m$ from Section~\ref{sec:sqiet}, the
left-hand subinterval of $I_\lambda$ always maps to the right-hand
side of some $I_{\lambda'}$ and similarly the right-hand subinterval
maps to the left-hand side of some $I_{\lambda''}$ where
$\lambda'' = vhv^{-1}(\lambda')$.  This is true for every slope $m$,
and so the only part of the combinatorial data of the interval
exchange that may change is the number of squares to cross over before
a point on the base exits the top of a cylinder.  See
Figure~\vref{fig:sqfirstreturn}.  Let $M$ denote the number of
vertical edges (labeled $H$) crossed when transitioning from the
left-hand side of $I_{\lambda}$ to the right-hand side of
$I_{\lambda'}$ In Figure~\ref{fig:sqfirstreturn}, for example,
$M = 2$.  In terms of Sturmian sequences, $M$ is the minimal number of
$H$'s between two $V$'s in the cutting sequence, regardless of which
letter is isolated.  The value of $M$ is determined by the slope $m$,
and considering that the line emitted from $(0, 0)$ with slope $m$
intersects the line $y = 1$ at the point $(\sfrac{1}{m}, 1)$, it is
easy to see that $M = \left\lfloor \frac{1}{m} \right\rfloor$.

Notice that if $V$'s are isolated (meaning $m < 1$), then the
blocks of $H$'s between $V$'s have length $M$ or $M + 1$.  If $H$'s
are isolated ($m > 1$, and so $M = 0$), then the number of $H$'s
between $V$'s is either $0$ or $1$.  As blocks of consecutive $H$'s
may have length $M$ or $M + 1$, we will refer to $M$ as the
\emph{length} of the Sturmian sequence.

The interval exchange $T_m$ is completely determined by the slope $m$;
in particular, writing $\frac{1}{m} = M + \theta$, the value of
$\theta$ determines the length data of $T_m$ and the value of $M$
determines the combinatorial data.  Representing this combinatorial
data as a pair of permutations $(\pi_0, \pi_1)$ for a given $T_m$ the
permutation $\pi_0$ is the same for all values of $M$ while $\pi_1$
may change.  For this reason we will write $\pi_1^M$ to mean the
permutation associated to the interval exchange $T_m$ when
$\left\lfloor \frac{1}{m} \right\rfloor = M$.

The combinatorial information of $(\pi_0, \pi_1^M)$ can be described
as a graph $\Gamma^M = \Gamma^M_{h,v}$.  This will be a labeled,
directed multigraph with vertex set $\Lambda$, each vertex
having two outgoing edges labeled $L$ and $R$, and two incoming edges
labeled $L$ and $R$.  We will call such a multigraph a
\emph{2-oriented graph}.  The edges from vertex $\lambda$ in graph
$\Gamma^M$ are given by
\[
  vh^M(\lambda) \xleftarrow{\quad L \quad} \lambda
  \xrightarrow{\quad R \quad} vh^{M+1}(\lambda).
\]
If $\lambda$ belongs to a cycle of $h$ with length $\ell$, then
$h^M(\lambda) = h^{M + k\ell}$ for all $k$.  Letting $h_1$, $h_2$,
..., $h_p$ denote the cycles of $h$, and $|h_i|$ the length of cycle
$h_i$, there are only
\[
  D = \lcm(|h_1|, |h_2|, ..., |h_p|)
\]
possible choices of combinatorial data $(\pi_0, \pi_1^M)$.  To see
this, notice that if $M > D$, say $M = pD + r$ for some
$0 \leq r < M$, then for each cycle $h_i$ we have
\[
  h_i^M = h_i^{pD + r} = \left(h_i^D\right)^p \cdot h_i^r = h_i^r
\]
since $h_i^D = \mathrm{id}$ as $D$ is a multiple of the length of
cycle $h_i$.  This is true for each cycle $h_i$ and so $h^M = h^r$.
As $T_m$ takes points on the left-hand side of the base of square
$\lambda$ to square $vh^M(\lambda)$, and points on the right-hand side
of the base of square $\lambda$ to square $vh^{M+1}(\lambda)$, the
value of $M$ only matters modulo $D$, and so there are only $D$ choices for
the combinatorial data of $T_m$.

As an example, consider the flow on the surface of
Figure~\ref{fig:sqtiled} in a direction with slope $m$ in the interval
$(\sfrac{1}{3}, \sfrac{1}{2})$.  For such a slope we will have
$M = \left\lfloor \frac{1}{m} \right\rfloor = 2$ and so consider the
graph $\Gamma^2$ where for the vertex labeled $\lambda$ has an
outgoing edge to $vh^2(\lambda)$ labeled $L$, and an outgoing edge to
$vh^3(\lambda)$ labeled $R$ as in Figure~\ref{fig:gamma}.  

\begin{figure}[h!]
  \centering
    \begin{tikzpicture}[node distance=2cm]
      \tikzstyle{vertex}=[circle,minimum size=1em,draw]
      \node[vertex] (s5) {$5$};
      \node[right of=s5,vertex] (s3) {$3$};
      \node[right of=s3,vertex] (s1) {$1$};
      \node[below of=s5,vertex] (s4) {$4$};
      \node[right of=s4,vertex] (s6) {$6$};
      \node[right of=s6,vertex] (s2) {$2$};

      \path[->,>=stealth,font=\tiny]
      (s1) edge[bend left=35] node[right] {$L$} (s2)
      (s1) edge[bend left=20] node[left] {$R$} (s2)
      (s2) edge node[below] {$L$} (s6)
      (s2) edge[bend left] node[left] {$R$} (s1)
      (s3) edge node[above] {$L$} (s1)
      (s3) edge[bend left=10] node[below] {$R$} (s5)
      (s4) edge[bend left=10] node[left] {$L$} (s5)
      (s4) edge[bend left=10] node[above] {$R$} (s6)
      (s5) edge[bend left=10] node[above] {$L$} (s3)
      (s5) edge[bend left=10] node[right] {$R$} (s4)
      (s6) edge[bend left=10] node[below] {$L$} (s4)
      (s6) edge node[right] {$R$} (s3)
      ;
  \end{tikzpicture}
  \caption{The graph $\Gamma^2$ associated to the square-tiled surface
    with permutations $h = (1)(2 \, 3 \, 4)(5 \, 6)$ and
    $v = (1 \, 2)(3 \, 5)(4 \, 6)$.}
  \label{fig:gamma}
\end{figure}
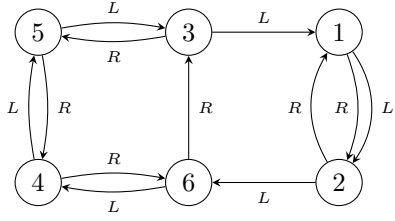

Each edge of this graph represents a possible transition in symbolic
trajectories of $T_m$: the edge labeled $L$ out of $\lambda$ and
ending at $vh^M(\lambda)$ corresponds to points in the
$[0, 1 - \theta)$-portion of $I_\lambda$ being mapped to
$I_{vh^M(\lambda)}$, and similarly for the edge labeled $R$.  Hence
every symbolic trajectory of $T_m$ determines an infinite walk on
$\Gamma^M$.  We will relate cutting sequences to symbolic trajectories
of the interval exchanges, and so infinite walks on these graphs.

\subsection{Combinatorial lifts and symbolic trajectories}
Let $(\lambda_n, \epsilon_n)$ be a combinatorial lift of some Sturmian
sequence.  We will convert $(\lambda_n, \epsilon_n)$ to a sequence
which we will show is the symbolic trajectory of an interval exchange
on the square-tiled surface.  First we define a subsequence $\mu_i =
\lambda_{n_i}$ of $(\lambda_n)$ by deleting those $\lambda_n$ where
$\epsilon_n = H$.  Now define a sequence $(\sigma_i) \in \Sigma^\Z = \{L,
R\}^{\Z}$ by considering the distance from $\lambda_{n_i}$ and
$\lambda_{n_{i + 1}}$:
\[
  \sigma_i = \begin{cases}
    L & \text{ if } n_{i + 1} - n_i = M+1 \\
    R & \text{ if } n_{i+1} - n_i = M+2
  \end{cases}
\]
where $M$ is the length (i.e., minimal number of $H$'s between two
$V$'s -- this may be zero) of the Sturmian sequence $\epsilon$.

We claim this new sequence $(\mu_n, \sigma_n) \in (\Lambda \times
\Sigma)^{\Z}$ is a symbolic trajectory for an interval exchange with
combinatorial data $(\pi_0, \pi_1^M)$.  To see this we need to verify
the Ferenczi-Zamboni conditions.  We will establish this through a
series of lemmas describing properties of $(\mu_n, \sigma_n)$ and the
graph $\Gamma^M$ which encodes the combinatorial data.

\begin{lemma}
  \label{lemma:LRSturmian}
  If $(\lambda_n, \epsilon_n)$ is a combinatorial lift of a Sturmian
  sequence $\epsilon$, then the sequence $\sigma$ defined above is a
  Sturmian sequence in the letters $L$ and $R$.  If $\epsilon$ is
  almost symmetric, then so is $\sigma$.
\end{lemma}
\begin{proof}
  Let $\gamma$ be the geodesic on the square torus represented by the
  Sturmian sequence $\epsilon$, let $I$ be the interval at the base of
  this square, and $R : I \to I$ the circle rotation (interval
  exchange transformation) given by the first-return map to the base
  $I$ of the square under the geodesic flow in the direction of
  $\gamma$.  This is an interval exchange on two intervals; we label
  the first, left-hand interval $L$, and the second, right-hand
  interval $R$.  The sequence $\sigma$ is precisely the sequences of
  $L$'s and $R$'s obtained by walking along $\gamma$ and recording an
  $L$ when $\gamma$ intersects the left-hand interval and an $R$ when
  $\gamma$ intersects the right-hand interval.  Since this sequence of
  $L$'s and $R$'s records the orbit of a point under an irrational
  rotation (irrational since $\epsilon$ is Sturmian hence the slope of
  the geodesic is not a rational number), $\sigma$ is itself a
  Sturmian sequence.  See Figure~\ref{fig:LsRsderiv} which shows the
  two different, parallel geodesic segments.  (The picture is drawn on
  a portion of the universal cover of the torus to make the picture
  easier to understand.)  Notice that the geodesic $\gamma_L$ which
  passes through the left-hand interval at the base of the first
  square intersects two vertical lines, while the geodesic $\gamma_R$
  which passes through the right-hand interval crosses three vertical
  lines.  That is, the corresponding cutting sequence is $VHHV$ for
  $\gamma_L$ and $VHHHV$ for $\gamma_R$.  In constructing $\sigma$ we
  are simply using the number of $H$'s (vertical line segments
  crossed) to determine if the geodesic previously passed through a
  left- or right-hand interval on the base of the square.

  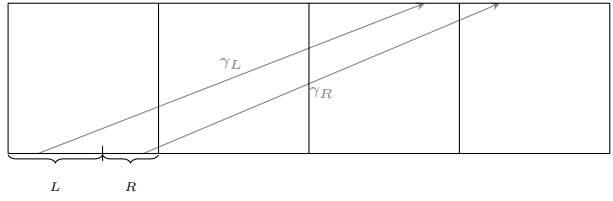
\begin{figure}[h!]
    \centering
    \begin{tikzpicture}[scale=2,font=\tiny]
      \draw[->,>=stealth,gray] (0.2, 0) -- node[above] {$\gamma_L$} (2.7696, 1);
      \draw[->,>=stealth,gray] (0.9, 0) -- node[below] {$\gamma_R$}
      (3.2696, 1);
      \draw (0.628, -0.05) -- (0.628, 0.05);
      \begin{scope}[yshift=4]
      \draw[underbrace style] (0, 0) -- node[underbrace text style]
        {$L$} (0.628, 0);
      \draw[underbrace style] (0.628, 0) -- node[underbrace text style]
      {$R$} (1, 0);
      \end{scope}
      \draw[step=1cm] (0, 0) grid (4, 1);
    \end{tikzpicture}
    \caption{The distance between two consecutive $V$'s in the cutting
      sequence of a geodesic determines, and is determined by, whether 
      the geodesic enters the left- or right-hand side of the base of
      the square.}
    \label{fig:LsRsderiv}
  \end{figure}

  If $V$ is the isolated character in an almost symmetric Sturmian
  sequence $\epsilon$, then the sequence has the form
  \[
    \overline{\omega} H V^M VH V^M H \omega \quad \text{ or } \quad
    \overline{\omega} H V^M HV V^M H \omega.
  \]
  with $\overline{\omega}$ being the same sequence as $\omega$, but in
  reverse order.  The constructed sequence $\sigma$ thus has the form
  \[
    \overline{\tau} R L^M R L L^M R \tau \quad \text{ or } \quad
    \overline{\tau} R L^M L R L^M R \tau.
  \]
  where $\tau$ is obtained by applying the rule for determining
  $\sigma_i$ above to the sequence $\omega$.

  Similarly, if $\epsilon$ is an almost symmetric Sturmian sequence
  but with isolated character $V$, then $\epsilon$ has the form
  \[
    \overline{\omega} V H^M VH H^M V \omega \quad \text{ or } \quad
    \overline{\omega} V H^M H V H^M V \omega,
  \]
  and the corresponding $\sigma$ sequence has the form
  \[
    \overline{\tau} LR \tau \quad \text{ or } \quad \overline{\tau} RL \tau.
  \]
  In either case,  we exactly have that $\sigma$ is almost symmetric. 
\end{proof}

\begin{lemma}
  \label{lemma:stronglyconnected}
  Each $\Gamma^M$ graph is strongly connected.
\end{lemma}
\begin{proof}
  As the edges from vertex $\lambda$ go to $vh^M(\lambda)$ and
  $vh^{M+1}(\lambda)$, we must show that there is some composition of
  powers of $vh^M$ and $vh^{M+1}$ which sends any given $\lambda_i$ to
  any $\lambda_j$.  Since $\langle h, v\rangle$ acts transitively on $\Lambda$,
  it suffices to show that $h$ and $v$ can be formed by compositions
  of powers of $vh^M$ and $vh^{M+1}$.

  Since $\Lambda$ has finitely-many elements, there exists some $p >
  0$ such that $(vh^M)^p = \id$ and so $(vh^M)^{p-1} = (vh^M)^{-1}$.
  Thus
  \[
    (vh^M)^{p-1} vh^{M+1} = h^{-M} v^{-1} vh^{M+1} = h.
  \]
  Similarly, there exists a $q$ such that $h^q = h^{-1}$ and so
  \[
    vh^M ((vh^M)^{p-1} vh^{M+1})^{qM} = vh^M (h^q)^M = vh^M h^{-M} = v.
  \]
\end{proof}

\begin{corollary}
  The permutations $(\pi_0, \pi_1^M)$ are irreducible.
\end{corollary}
\begin{proof}
  If the permutations were not irreducible, then the corresponding
  interval exchange on the surface would have a proper subset of
  intervals which are left invariant by the interval exchange.  These
  invariant subsets of intervals would correspond to connected
  components of the graph $\Gamma^M$, but
  Lemma~\ref{lemma:stronglyconnected} shows there is only one such
  component.
\end{proof}

We will say that an edge of $\Gamma^M$ is a \emph{bad edge} if
traversing that edge could correspond to hitting a conical singularity
on the square-tiled surface.  Each bad square of the surface gives
two bad edges on each $\Gamma^M$: if $\lambda$ is a bad square, then
the edge labeled $L$ with target $v(\lambda)$ and the edge labeled $R$
with target $vh(\lambda)$ are bad edges.

Notice that $(\mu_n, \sigma_n)_{n \in \Z}$ determines a walk on the graph
$\Gamma^M$, where $\mu_{n+1}$ is obtained by traveling from $\mu_n$
along edge $\sigma_n$.  We will say that the walk $(\mu_n, \sigma_n)$ is
\emph{almost symmetric about a bad edge} if there exists an $N$ such
that $\sigma_{N+k} = \sigma_{N-k-1}$ for each $k \geq 1$ and the corresponding
edge
\[
  \mu_N \xrightarrow{\quad \sigma_N \quad} \mu_{N+1}
\]
is a bad edge of $\Gamma^M$.

We now claim that \emph{Sturmian walks}, walks where the sequence of
crossed edges form a Sturmian sequence, on a $\Gamma^M$ graph are in
one-to-one correspondence with symbolic trajectories of interval
exchanges $T_m$ on the corresponding square-tiled surface.

Notice that symbolic trajectories of the interval exchange on the
bases of squares in a square-tiled surface give Sturmian walks on the
corresponding $\Gamma^M$ graph.  This is simply because the sequences
of edges labeled $L$ and $R$ coincide with the moments when the
corresponding geodesic on the torus intersects the left- and right-hand
intervals on the base of the square.  As we are considering irrational
directions, these $L$ and $R$ subintervals are given by an irrational
rotation, and so produce a Sturmian sequence.  To show that each such
Sturmian walk corresponds to a symbolic trajectory of the interval
exchange we need to verify the Ferenczi-Zamboni conditions.  We first
introduce some notation for the admissible prefixes and suffixes of a
word.

Given a word
$w = (\nu_0, s_0) \, (\nu_1, s_1) \, \cdots \, (\nu_k, s_k) \in
(\Lambda \times \Sigma)^*$, let $\left( p_n(w) \right)_{n \in \Z}$ be
the sequence which gives the locations of $w$ in the sequence
$(\mu_n, \sigma_n) \in (\Lambda \times \Sigma)^\Z$. That is, $p_n(w)$ satisfies
\begin{align*}
  (\mu_{p_n(w)}, \sigma_{p_n(w)})
  &= (\nu_0, s_0), \\
  (\mu_{p_n(w)+1}, \sigma_{p_n(w)})
  &= (\nu_1, s_1), \\
  &\vdots\\
  (\mu_{p_n(w) + k},\sigma_{p_n(w)}) 
  &= (\nu_k, s_k).
\end{align*}

For example, suppose $\Lambda = \{1, 2\}$ and consider the word
\[
  w = (1, L) \, (2, R) \, (1, L) \, (2, R).
\]
If the sequence $(\mu_n, \sigma_n)$ contained
\[
  ... (2, L) \, (1, L) \, (2, R) \, (1, L) \, (2, R) \, (1, L) \, (2,
  R) \, (1, L) \, (2, R) \, (1, R) ...
\]
with the initial $(2, L)$ corresponding to $(\mu_0, \sigma_0)$, then
we would have $p_0(w) = 1$, $p_1(w) = 3$, $p_2(w) = 5$.  Notice that
we allow different copies of $w$ to overlap one another in our
definition of $p_n(w)$.

We will adopt the convention that $p_0(w)$ is the first occurrence of
$w$ starting at an index $n \geq 0$ in $(\mu_n, \sigma_n)$, though
this particular choice of starting point will not affect what is to follow.

Following the notation of \cite{FerencziZamboni}, we let $A(w)$ denote
the set of characters which occur immediately before $w$ appears
in $(\mu, \sigma)$, and $D(w)$ denotes the set of characters which occur
immediately after $w$:
\begin{align*}
  A(w) &= \left\{(\ell, s) \, \big| \, \ell = \mu_{p_n(w)-1}, \, s
         = \sigma_{p_n(w)-1} \text{ for some } n \right\} \text{ and } \\
  D(w) &= \left\{(\ell, s) \, \big| \, \ell = \mu_{p_n(w)+k+1}, \, s
         = \sigma_{p_n(w)+k+1} \text{ for some } n \right\}.
\end{align*}

In the example above we would thus have $(2, L), (2, R) \in A(w)$ and
$(1, L), (1, R) \in D(w)$.  In principle the set of prefixes and
suffixes, which were called \emph{arrival} and \emph{departure} sets
in \cite{FerencziZamboni}, for a general language could be quite large
and complicated sets of letters.  The following lemma shows that these
sets are actually very simple for the languages we are considering.

\begin{lemma}
  \label{lemma:adintervals}
  The sets $A(w)$ and $D(w)$ each contain at most two characters, $A(w)$
  is a $\pi_1^M$-interval, and $D(w)$ is a $\pi_0$-interval.
\end{lemma}
\begin{proof}
  It is clear that $A(w)$ has two characters since there are only two
  edges in $\Gamma^M$ ending at $\nu_0$: one labeled $L$ and one labeled
  $R$.  From the definition of the graph, we can see that the edge
  labeled $L$ ending at $\nu_0$ starts at vertex $h^{-M}v^{-1}(\nu_0)$ and
  the edge labeled $R$ ending at $\nu_0$ starts at vertex
  $h^{-(M+1)}v^{-1}(\nu_0)$.  It is also easy to see from the definition
  of the permutation $\pi_1^M$ that these two symbols,
  $(h^{-M}v^{-1}(\nu_0), L)$ and $(h^{-(M+1)}v^{-1}(\nu_0), R)$, are
  $\pi_1^M$-consecutive:
  \begin{align*}
    \pi_1^M(h^{-M}v^{-1}(\nu_0), L) &= 2vh^M h^{-M}v^{-1}(\nu_0) = 2\nu_0\\
    \pi_1^M(h^{-(M+1)}v^{-1}(\nu_0), R) &= 2vh^{M+1} h^{-(M+1)}
    v^{-1}(\nu_0) = 2\nu_0 - 1.
  \end{align*}
  The proof that $D(w)$ is a $\pi_0$-interval and that it contains at
  most two elements is similar.
\end{proof}

\begin{lemma}
  \label{lemma:minlseq}
  If $(\sigma_n)$ is a Sturmian sequence in $\{L, R\}$, any
  $(\mu_n, \sigma_n)$ satisfying
  $\mu_{n+1} = \sigma_n \cdot \mu_n$ is minimal.
\end{lemma}
\begin{proof}
  We need to show that for each finite word of
  $(\mu_n, \sigma_n)$, there exists some bound $M$ such that every
  word of length $M$ has the given finite word as a factor.  Shifting
  if necessary, we may assume the finite word is
  \[
    (\mu_0, \sigma_0) \, (\mu_1, \sigma_1) \, \cdots (\mu_k, \sigma_k).
  \]
  As $(\sigma_n)$ is a Sturmian sequence, there exists some $N_0$ such
  that all words of $(\sigma_n)$ of length $N_0$ have $\sigma_0 \,
  \sigma_1 \, \cdots \, \sigma_k$ as a factor.  Letting $(\tau_i)$ be the
  sequence of integers satisfying
  \[
    \sigma_{\tau_i} = \sigma_0, \, \sigma_{1+\tau_i} = \sigma_1, \,
    \cdots \, , \sigma_{k+\tau_i} = \sigma_k,
  \]
  with $\tau_0 = 0$, the $(\tau_i)$ sequence has bounded gaps:
  $\tau_{i+1} - \tau_i < N_0$.  To show minimality of
  $(\mu_n, \sigma_n)$, it suffices to show that there exists an
  $M$ such that $\mu_{\tau_i} = \mu_0$ for some $i < M$.  To show
  this it will be helpful to think about the sequence $(\mu_n,
  \sigma_n)$ as determining a walk on a graph $\Gamma^M$.

  Let $\psi_0$ be the empty word and set $\psi_i = \sigma_{\tau_i - 1} \,
  \sigma_{\tau_i - 2} \cdots \sigma_{\tau_{i-1}}$.  We can interpret
  the $\psi_i$'s both as words of $(\sigma_n)$, and as permutations of
  $\Lambda$ with $\psi_0$ being the identity.  Notice $\mu_{\tau_1}
  = \psi_1(\mu_0)$ and $\mu_{\tau_i} =
  \psi_i(\mu_{\tau_{i-1}}) = \psi_i \, \psi_{i-1} \, \cdots \,
  \psi_1(\mu_0)$.

  The first few edges crossed by the walk $(\mu, \sigma)$ are as follows:
  
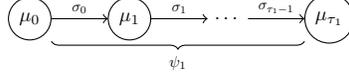
\begin{center}
    \begin{tikzpicture}[scale=0.75, transform shape]
      \node[circle,draw] (L0) {$\mu_0$};
      \node[circle,draw,right=1cm of L0] (L1) {$\mu_1$};
      \node[right=1cm of L1] (D1) {$\cdots$};
      \node[circle,draw,right=1cm of D1] (LT1) {$\mu_{\tau_1}$};

      \path[->]
        (L0) edge node[scale=0.75,above]{$\sigma_0$} (L1)
        (L1) edge node[scale=0.75,above]{$\sigma_1$} (D1)
        (D1) edge node[scale=0.75,above]{$\sigma_{{\tau_1}-1}$} (LT1)
        ; 
      \draw[underbrace style] (L0.east) -- node[underbrace text style] {$\psi_1$} (LT1.west);
    \end{tikzpicture}
  \end{center}
  where the walk from $\mu_0$ to $\mu_{\tau_1}$ has at most
  $N_0$ characters.  If $\mu_{\tau_1} \neq \mu_0$, then
  consider a longer portion of the walk until $\psi_1$ appears again,
  ending at some $\mu_{\tau_{i_1}}$.  Since $(\sigma_n)$ is
  Sturmian, there exists some $N_1$ so that all words of length $N_1$
  have $\psi_1$ as a factor.  Thus the distance from $\mu_0$ to
  $\mu_{\tau_{i_1}}$ is at most $N_0 + N_1$.  If $\mu_0$ does
  not appear in $\mu_{\tau_1}$, $\mu_{\tau_2}$, ...,
  $\mu_{\tau_{i_1}}$, notice that $\mu_{\tau_{i_1}}$ can not equal
  $\mu_0$ (by assumption), but can also not equal
  $\mu_{\tau_1}$ since
  $\mu_{\tau_{i_1}} = \psi_1(\mu_{\tau_{(i_1 - 1)}})$ and
  $\mu_{\tau_{(i_1 - 1)}} \neq \mu_0$.
  
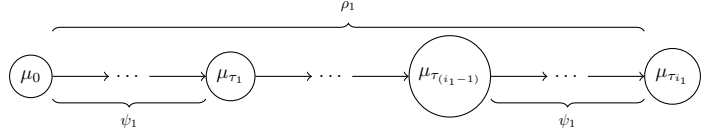
\begin{center}
    \begin{tikzpicture}[scale=0.75,transform shape]
      \node[circle,draw] (L0) {$\mu_0$};
      \node[right=1cm of L0] (D1) {$\cdots$};
      \node[circle,draw,right=1cm of D1] (LT1) {$\mu_{\tau_1}$};
      \node[right=1cm of LT1] (D2) {$\cdots$};
      \node[circle,draw,right=1cm of D2] (LTI1L) {$\mu_{\tau_{(i_1
            - 1)}}$};
      \node[right=1cm of LTI1L] (D3) {$\cdots$};
      \node[circle,draw,right=1cm of D3] (LTI1) {$\mu_{\tau_{i_1}}$};

      \path[->]
        (L0) edge (D1)
        (D1) edge (LT1)
        (LT1) edge (D2)
        (D2) edge (LTI1L)
        (LTI1L) edge (D3)
        (D3) edge (LTI1)
        ; 
      \draw[underbrace style] (L0.east) -- node[underbrace text style] {$\psi_1$} (LT1.west);
      \draw[underbrace style] (LTI1L.east) -- node[underbrace text
      style] {$\psi_1$} (LTI1.west);
      \draw[overbrace style] (L0.east) -- node[overbrace text style]
      {$\rho_1$} (LTI1.west);
    \end{tikzpicture}
  \end{center}
  Let $\rho_1$ be the product of all the symbols (permutations)
  between $\mu_0$ and $\mu_{\tau_{i_1}}$.
  As $(\sigma_n)$ is Sturmian, there exists an $N_2$ such that all
  words of length $N_2$ contain $\rho_1$ as a factor.  Consider now
  the shortest portion of our walk which contains a second factor of
  $\rho_1$ and ends at some $\mu_{\tau_{i_2}}$.
  
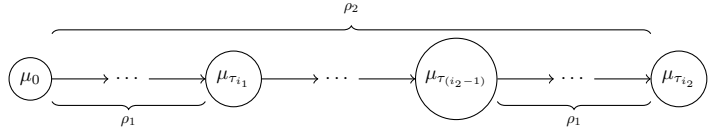
\begin{center}
    \begin{tikzpicture}[scale=0.75,transform shape]
      \node[circle,draw] (L0) {$\mu_0$};
      \node[right=1cm of L0] (D1) {$\cdots$};
      \node[circle,draw,right=1cm of D1] (LTI1) {$\mu_{\tau_{i_1}}$};
      \node[right=1cm of LTI1] (D2) {$\cdots$};
      \node[circle,draw,right=1cm of D2] (LTI2L) {$\mu_{\tau_{(i_2 - 1)}}$};
      \node[right=1cm of LTI2L] (D3) {$\cdots$};
      \node[circle,draw,right=1cm of D3] (LTI2) {$\mu_{\tau_{i_2}}$};

      \path[->]
        (L0) edge (D1)
        (D1) edge (LTI1)
        (LTI1) edge (D2)
        (D2) edge (LTI2L)
        (LTI2L) edge (D3)
        (D3) edge (LTI2)
        ; 
      \draw[underbrace style] (L0.east) -- node[underbrace text style] {$\rho_1$} (LTI1.west);
      \draw[underbrace style] (LTI2L.east) -- node[underbrace text
      style] {$\rho_1$} (LTI2.west);
      \draw[overbrace style] (L0.east) -- node[overbrace text style]
      {$\rho_2$} (LTI2.west);
    \end{tikzpicture}
  \end{center}
  If no $\mu_1$ through $\mu_{\tau_{i_2}}$ equals $\mu_0$,
  then $\mu_{\tau_{i_2}}$ is not equal to $\mu_0$,
  $\mu_{\tau_1}$, or $\mu_{\tau_{i_1}}$, and the number of
  steps from $\mu_0$ to $\mu_{\tau_{i_2}}$ is at most $N_0 +
  N_1 + N_2$.  Letting $\rho_2$ denote the transitions from
  $\mu_0$ to $\mu_{\tau_{i_2}}$, there exists some $N_3$ such
  that all words of length at least $N_3$ contain a factor of
  $\rho_2$, and we can consider the shortest portion of our walk
  starting from $\mu_0$ and containing two disjoint factors of
  $\rho_2$.

  We continue in this way building a sequence of symbols
  $\mu_{\tau_{i_k}}$ which can not take on $k + 1$ different
  symbols of $\Lambda$.  If we repeat this procedure up until
  $\mu_{\tau_{i_d}}$, then this symbol can not take on any value
  in the alphabet.  Hence at some point $\mu_0$ must repeat, and
  this must happen within $N_0 + N_1 + N_2 + \cdots + N_d$ steps.
\end{proof}

\begin{corollary}
  If $(\lambda_n, \epsilon_n)$ is a combinatorial lift of a Sturmian
  sequence, then the associated sequence $(\mu_n, \sigma_n)$ is
  minimal.
\end{corollary}

\begin{lemma}
  \label{lemma:allchars}
  If $(\sigma_n)$ is a Sturmian sequence in $\Sigma^\Z$, each $(\mu_n,
  \sigma_n)$ satisfying $\mu_{n+1} = \sigma_n \cdot \mu_n$ contains
  every element of $\Lambda \times \Sigma$.
\end{lemma}
\begin{proof}
  Thinking of $(\mu_n, \sigma_n)$ as a walk on a graph
  $\Gamma = \Gamma^M$, the lemma claims that every edge of the graph
  is crossed by the walk.  Suppose this was not the case and there was
  some edge the walk does not cross.  We will perform an iterative
  procedure replacing our graph with a \emph{derived graph} $\Gamma'$
  so that the uncrossed edge of $\Gamma$ becomes a vertex which the
  corresponding \emph{derived walk} $(\mu_n', \sigma_n')$ avoids.  The
  derived graph will be a 2-oriented graph with the same number of
  vertices and edges as $\Gamma$, and so a walk which avoids one
  vertex implies there are at least three edges of $\Gamma'$ which the
  derived walk does not cross.  Repeating the derivation procedure
  produces a graph $\Gamma''$ and a walk $(\mu_n'', \sigma_n'')$ which
  avoids two vertices, and at least four edges are not crossed.
  Since our graphs all have the same finite number of vertices and
  edges, this procedure will after finitely-many steps show that our
  walk must not cross \emph{any} edge. 

  We define the derived graph $\Gamma'$ and derived walk $(\mu_n',
  \sigma_n')$ as follows.  Suppose that the edge of $\Gamma$ which is
  not crossed is
  \[
    \lambda \xrightarrow{\sigma} \lambda'.
  \]
  Let $\tau$ denote the element of $\Sigma$ which is not the $\sigma$
  above.  Now consider blocks of $(\sigma_n)$ which consist of blocks
  of consecutive $\tau$'s and end in $\sigma$.  There will be two such
  blocks, but the possible blocks depend on the length $M$ of the
  Sturmian sequence and whether $\sigma$ or $\tau$ is the isolated
  character.  We will refer to the two possible blocks as $L'$ and
  $R'$.

  If $\sigma$ is the isolated character of $(\sigma_n)$, then the
  blocks ending in $\sigma$ have the form
  \[
    L' = \tau^M\sigma \text{ or } R' = \tau^{M+1}\sigma.
  \]
  If $\tau$ is the isolated character, then the blocks have the form
  \[
    L' = \sigma \text{ or } R' = \tau\sigma.
  \]
  In either case we may interpret $L'$ and $R'$ as permutations of
  $\Lambda$ and consider the graph $\Gamma'$ with vertices the
  characters of $\Lambda$ and edges
  \[
    L'(\lambda) \xleftarrow{L'} \lambda \xrightarrow{R'} R'(\lambda).
  \]
  Notice that $\Gamma'$ is a 2-oriented graph with the same number of
  edges and vertices as $\Gamma$.  Furthermore, the proof of
  Lemma~\ref{lemma:stronglyconnected} applies equally well to
  $\Gamma'$, and so $\Gamma'$ is strongly connected.
  
  The derived walk $(\mu_n', \sigma_n')$ consists of the sequence
  $\sigma_n'$ of $L'$ and $R'$ obtained by replacing the blocks in $(\sigma_n)$ above
  with $L'$ and $R'$, and setting $\mu_0' = \mu_0$ and $\mu_{n+1}' =
  \sigma_n' \cdot \mu_n'$.

  Notice that $(\sigma_n')$ is a Sturmian sequence.  In the case that
  $\tau$ is the isolated character, replacing blocks of $\tau \sigma$
  with $R'$ and blocks of $\sigma$ with $L'$ is the same as deriving
  the Sturmian sequence (removing a single $\sigma$ from blocks of
  consecutive $\sigma$'s), which gives a Sturmian sequence, and then
  doing a character-by-character replacement by substituting $\tau$
  with $R'$ and $\sigma$ with $L'$, which is obviously a Sturmian
  sequence.  In the case that $\sigma$ is the isolated character,
  $(\sigma_n')$ is obtained by deriving $(\sigma_n)$ $M$ times so that
  $\tau$ becomes the isolated character, and then performing the
  operation previously described.

  The derived walk $(\mu_n', \sigma_n')$ on $\Gamma'$ is an
  acceleration of the walk $(\mu_n, \sigma_n)$ on $\Gamma$ where we
  consider longer pieces of the original walk ending in $\sigma$.
  Since $\lambda \xrightarrow{\sigma} \lambda'$ is not crossed in the
  original walk, there are no blocks crossing such an edge.  Since the
  other incoming edge around $\lambda'$ must be labeled by $\tau$, the
  derived walk avoids the vertex $\lambda'$ of $\Gamma'$.  As there
  are two edges which enter $\lambda'$, the derived walk must not
  cross either of these edges, and can not cross either of the
  outgoing edges either.  It may happen that one of the incoming edges
  may also be an outgoing edge (i.e., there may be a loop at
  $\lambda'$).  Thus there are at least three forbidden edges of
  $\Gamma'$ the derived walk can not cross.  There can not be fewer
  than three edges because this would violate the strong connectedness
  of the graph.

  As $(\sigma_n')$ is a Sturmian sequence, one character, say
  $\sigma'$, is isolated and we can repeat the above procedure by
  considering blocks of $(\sigma_n')$ which end in $\sigma'$ to
  produce a second derived graph $\Gamma''$ and second derived walk
  $(\mu_n'', \sigma_n'')$.  Again, $\Gamma''$ will be a
  strongly-connected 2-oriented graph and $(\sigma_n'')$ will be a
  Sturmian sequence.  As $(\sigma_n')$ avoids two edges, the walk
  $(\mu_n'', \sigma_n'')$ avoids at least four edges.  Loops may occur
  and so the number of avoided edges does not double each time.
  However, the strong connectedness of the graph prevents us from
  having a connected component of avoided edges and so the number of
  forbidden edges must grow by at least one.

  Each time we replace a graph and a walk on that graph with its
  derivation in this manner, the number of avoided edges grows by at
  least one.  Since there are only finitely-many edges, however, after
  finitely-many derivations, our walk, which is given by a Sturmian
  sequence, can not cross any edge of the graph.  Thus there can not
  be any edge avoided by the graph, and so the original sequence
  $(\mu_n, \sigma_n)$ must contain each element of
  $\Lambda \times \Sigma$.
\end{proof}

\begin{proposition}
  \label{prop:sturmianwalk}
  A Sturmian walk on $\Gamma^M$ satisfies the Ferenczi-Zamboni
  conditions and thus is the symbolic trajectory of some interval
  exchange with combinatorial data $(\pi_0, \pi_1^M)$.
\end{proposition}
\begin{proof}
  Let $(\mu_n, \sigma_n) \in (\Lambda \times \Sigma)^{\Z}$ denote
  the walk.  We simply verify that the six conditions of Theorem~\ref{thm:FZ}
  are satisfied.  Condition (1) is ensured by
  Lemma~\ref{lemma:minlseq}.  Condition (2) is given by
  Lemma~\ref{lemma:allchars}.  Conditions (3) and (4) are given by
  Lemma~\ref{lemma:adintervals}.

  For condition (5), we remark that since $\sigma$ is a Sturmian
  sequence it represents the symbolic trajectory of a point under an
  irrational rotation and so must satisfy the Ferenczi-Zamboni
  conditions.  In particular, condition (5) holds on the
  factor $\sigma$ in our sequence since this is a Sturmian
  sequence: $\sigma$ is the symbolic trajectory of
  an interval exchange with intervals labeled by $\Sigma = \{L, R\}$
  with combinatorial data
  \begin{align*}
    \pi_0(L) &= 1 &     \pi_1(L) &= 2 \\
    \pi_0(R) &= 2 &     \pi_1(R) &= 1
  \end{align*}
  and length data $\ell_L = 1 - \theta$, $\ell_R = \theta$ where
  $\theta$ is the rotation parameter described earlier in
  Section~\ref{subsec:sturmian}.

  Now suppose that
  \[
    w = (\nu_0, s_0) \, (\nu_1, s_1) \, \cdots \, (\nu_k, s_k)
  \]
  is some finite subword of $(\mu, \sigma)$.  Let $s$ denote the
  Sturmian word $s_0 s_1 \cdots s_k$.  The admissible prefixes and
  suffixes of $w$ are then completely determined by the admissible
  prefixes and suffixes of $s$.
  In particular, there are only two points of concern in showing condition
  (5) holds, and both of these concerns are alleviated in light of
  this observation.

  \begin{enumerate}[1.]
  \item Suppose $A(w)$ is a singleton, say $(\xi, t)$ is the unique
    element of $A(w)$.  We must show that $D((\xi, t) \, w)$ is a singleton as
    well.  However, if $A((\xi, t) \, w)$ is a singleton, then so is
    $A(t \, s)$ and
    since the Ferenczi-Zamboni conditions hold for the sequence
    $(s_n)$, $D(t \, s)$ must be a singleton.  This then implies that
    $D((\xi, t) \, w)$ is a singleton, since the second component must
    be the only element in $D(ts)$ and the first one is determined by
    $\mu_{n+1} = \sigma_n \, \mu_n$.  (Note here that the
    inequalities in condition (5) of the Ferenczi-Zamboni conditions
    occur simultaneously.)
  \item
    Suppose now that
    \[
      A(w) = \left\{(R^{-1}(\nu_0), R), \, (L^{-1}(\nu_0), L)\right\},
    \]
    where $R^{-1} = h^{-(M+1)}v^{-1}$ and $L^{-1} = h^{-M} v^{-1}$.
    Then $A(s) = \{R, \, L\}$.  As
    $(R^{-1}(\nu_0), R) \leq_{\pi_1} (L^{-1}(\nu_0), L)$ and
    $(s_k \cdot \nu_k, L) \leq_{\pi_0} (s_k \cdot \nu_k, R)$, we must show that
    if $(s_k \cdot \nu_k, R) \in D((R^{-1}(\nu_0), R) \, w)$, then
    $D((L^{-1}(\nu_0), L) \, w) = \{(s_k \cdot \nu_k, R)\}$.  However,
    since the Sturmian sequence $(s_n)_{n \in \Z}$ satisfies the
    Ferenczi-Zamboni conditions, this follows immediately from the
    fact that if $A(s) = \{L, R\}$ and $R \in D(R \, s)$ then
    $D(L \, s) = \{R\}$.
  \end{enumerate}

  To show condition (6), we again remark that the Sturmian sequence
  $\sigma$ satisfies the Ferenczi-Zamboni conditions.  In particular,
  condition (6) in the case of $\sigma$ implies that if $w$ is a
  finite subword of $\sigma$, and $s, s' \in A(w)$ are distinct
  symbols, then $D(sw) \cap D(s'w)$ must be a singleton.  This
  trivially implies the same for the sequence $(\mu, \sigma)$.
\end{proof}

We are now able to combine the results above to establish our main
theorem, restated below for the convenience of the reader.

\begin{theorem}[The Characterization Theorem]
  \label{thm:characterization}
  Let $X$ be a square-tiled surface on $d$ squares determined by a
  pair of permutations $h$ and $v$ on $\Lambda = \{1, 2, ..., d\}$.
  Let $E = \{H, V\}$ be symbols for the edges of the unit square
  torus.  Label the left-hand edge of square $\lambda$ as
  $(\lambda, H)$ and the bottom edge as $(\lambda, V)$.  Then a
  biinfinite sequence
  $(\lambda_n, \epsilon_n) \in \left(\Lambda \times E\right)^{\Z}$ is
  the cutting sequence of an infinite geodesic on $X$ if and only if
  the following conditions are satisfied:
  \begin{enumerate}
  \item $(\lambda_n, \epsilon_n)$ is consistent with the gluings of
    the surface;
  \item $(\lambda_n, \epsilon_n)$ is either periodic, or is minimal
    but not almost symmetric around a bad square of the surface; and
  \item $\epsilon_n$ is the cutting sequence of a geodesic on the
    square torus.
  \end{enumerate}
\end{theorem}

The strategy of the proof will be to take a sequence satisfying the
three conditions of Theorem~\ref{thm:characterization} and associate
to that sequence the symbolic trajectory of an interval exchange on
the surface.  Such an association necessarily loses some information
as two different geodesics with different cutting sequences could give
rise to the same symbolic trajectory if those geodesics are related by
a Dehn twist.  The correspondence between cutting sequences and
symbolic trajectories is not one-to-one, but instead depends on a
choice which essentially keeps track of how many times Dehn twists
were applied.

\begin{proof}[Proof of Theorem~\ref{thm:characterization}]
  One direction of the proof is simple: it is clear that a cutting
  sequence $(\lambda_n, \epsilon_n)$ of a geodesic $\gamma$ on the
  surface satisfies the three conditions since the geodesic on $X$
  projects to a geodesic on the square torus $\mathbb{T}^2$.  By the
  Veech dichotomy, the geodesic flow in any given direction on $X$ or
  $\mathbb{T}^2$ is either periodic (if the flow is in a rational
  direction), or uniquely ergodic (if the flow is in an irrational
  direction).  As Lebesgue measure is preserved by flows in any
  direction, if the flow is uniquely ergodic the flow must in fact be
  minimal as it will enter any open ball on the torus.  Since the flow
  on any translation surface is the suspension over an interval
  exchange (coming from the first-return map of the geodesic flow to
  an appropriately chosen transversal geodesic segment), minimality of
  the flow on the surface implies minimality of the interval exchange
  which implies minimality of symbolic trajectories.  Thus
  non-periodic cutting sequences must come from geodesics in
  irrational directions and so are required to be minimal.

  For the converse, suppose from $(\lambda_n, \epsilon_n)$ we
  construct sequences $(\mu_n, \sigma_n)$ as above.  Then by
  Proposition~\ref{prop:sturmianwalk}, $(\mu_n, \sigma_n)$ is the
  symbolic trajectory of some interval exchange with the combinatorial
  data $(\pi_0, \pi_1^M)$ as determined by the surface.  We need to
  show now that the interval exchange on the bases of the squares, as
  described in Section~\ref{sec:sqiet}, is one of the interval
  exchanges for which $(\mu_n, \sigma_n)$ is a symbolic trajectory.
  By the proof of \cite[Theorem~2]{FerencziZamboni}, this means we need to
  exhibit a shift-invariant measure on the dynamical system described
  by $(\mu_n, \sigma_n)$ where the measure on each cylinder set agrees
  with the lengths of the subintervals of the interval exchange on the
  surface.

  We consider the interval exchange $T_m : I \times \Lambda \to I
  \times \Lambda$ corresponding to the flow with slope $m$ being the
  slope of the Sturmian sequence $\epsilon$ as described in
  Section~\ref{subsec:sturmian}.  
  
  Define a measure $\mu$ on the associated shift space by declaring the
  measure of a cylinder set to be the Lebesgue measure of the
  corresponding subinterval in our interval exchange.  Denoting the
  length of an interval $J$ by $|J|$, we are thus considering the
  measure given by $\mu([w]) = |I_w^0|$.  This clearly defines a
  measure, but we must check that the measure is shift invariant.

  Let $\Delta$ denote the shift, and let $T$ denote the interval
  exchange on the surface.  We need to show that for each cylinder
  $[w]$, $\mu([w]) = \mu(\Delta^{-1}([w]))$.  Notice
  \[
    \Delta^{-1}([w]) = \bigcup_{a \in A(w)} [aw]
  \]
  where again $A(w)$ is the set of allowable one-character prefixes of
  the word $w$.  We must show the length $\left| I_w^0 \right|$ of the
  subinterval $I_w^0$ equals the measure of the $\Delta$-preimage of $w$.

  We may express the measure of the preimage of $w$ as
  \begin{align*}
    \mu\left(\bigcup_{a \in A(w)} [aw]\right)
    &= \sum_{a \in A(w)} \left| I_{aw}^0 \right| \\
    &= \left|\bigcup_{a \in A(w)} \left[ I_a^0 \cap T^{-1}(I_w^0)
      \right] \right| \\
    &= \left| \left[ \bigcup_{a \in A(w)} I_a^0 \right] \cap
    T^{-1}(I_w^0) \right|\\
    &= \left| T\left( \left[ \bigcup_{a \in A(w)} I_a^0 \right] \cap
      T^{-1}(I_w^0) \right) \right| \\
    &= \left| \left[ \bigcup_{a \in A(w)} I_a^1 \right] \cap I_w^0 \right|.
  \end{align*}

  By the assumption that the Ferenczi-Zamboni conditions are
  satisfied, $A(w)$ is a $\pi_1^M$-interval, and so
  $\bigcup_{a \in A(w)} I_a^1$ is a subinterval of $I$.  Since
  elements of $A(w)$ are precisely the $a$ such that
  $I^0_{aw} \neq \emptyset$, the interval $\bigcup_{a \in A(w)} I_a^1$
  contains $I_w^0$ which implies shift-invariance of our measure
  $\mu$.  Thus $T_m$ is one of the interval exchanges for which the
  given sequence is a symbolic trajectory.
  
  For each choice of $M \in \{0, 1, 2, ..., D\}$, with $D$ the least
  common multiple of the lengths of cycles of the permutation $h$, and
  each choice of $k \in \mathbb{N}_0$, there is a one-to-one
  correspondence between symbolic trajectories of the interval
  exchange on our surface with combinatorial data $(\pi_0, \pi_1^M)$
  and cutting sequences $(\lambda_n, \epsilon_n)$ where the length
  modulo $D$ of the Sturmian sequence $(\epsilon_n)$ is $M$.

  To see
  this, we simply replace the characters in the sequence as follows:
  $(\mu, L)$ is replaced by
  \[
    (\mu, V) \, (h(\mu), H) \, (h^2(\mu), H) \, \cdots \,
    (h^{kD + M}(\mu), H), 
  \]
  and $(\mu, R)$ is replaced by
  \[
    (\mu, V) \, (h(\mu), H) \, (h^2(\mu), H) \, \cdots \,
    (h^{kD + M}(\mu), H) \, (h^{kD + M + 1}(\mu), H).
  \]

  The values of $M$ and $k$ have a very geometric
  interpretation.  Given any geodesic $\gamma$ that crosses the base of
  a square $\lambda$ on the surface, consider the geodesic $\delta$
  which is obtained by applying a multiple Dehn twist to the surface
  so that each horizontal cylinder formed by the squares tiling the
  surface (i.e., corresponding to each cycle of the permutation $h$)
  has its top and bottom edges fixed.  The geodesics $\gamma$ and
  $\delta$ certainly have different cutting sequences since $\delta$
  will cross more vertical edges labeled $H$ than $\gamma$ will.
  However, both geodesics have the same symbolic trajectory under the
  interval exchange on the bases of the squares.  See
  Figure~\ref{fig:twisted} for an example.  

  \begin{figure}[h!]
    \centering
    \begin{tikzpicture}
      \coordinate (s1) at (0, 2);
      \coordinate (s2) at (0, 1);
      \coordinate (s3) at (1, 1);
      \coordinate (s4) at (2, 1);
      \coordinate (s5) at (1, 0);
      \coordinate (s6) at (2, 0);

      \foreach \x in {1, 2, ..., 6} {
        \draw (s\x) rectangle ($(s\x) + (1, 1)$);
      }

      \draw[->,>=stealth,gray] (0.2, 1) -- (2.7696, 2);
      \draw[->,dashed,>=stealth,gray] (0.2, 1) -- (3, 1.33454);
      \draw[->,dashed,>=stealth,gray] (0, 1.33454) -- (3, 1.69298);
      \draw[->,dashed,>=stealth,gray] (0, 1.69298) -- (2.7696, 2);
    \end{tikzpicture}
    \caption{Though $\gamma$ (solid) and $\delta$ (dashed) give the same symbolic
      trajectory of the interval exchange on the surface, they have
      different cutting sequences.}
    \label{fig:twisted}
  \end{figure}
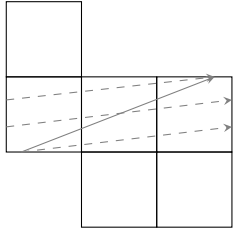

  That is, the sequence $(\mu_n, \sigma_n)$ does not contain enough
  information to determine a cutting sequence on the surface since two
  different geodesics, related by an appropriate Dehn twist, will have
  different cutting sequences but the same $(\mu_n, \sigma_n)$
  sequence.  The $k$ tells us how many times the Dehn twists are
  applied and are necessary for recovering the cutting sequence from the
  symbolic trajectory of the interval exchange.

  Finally, note that a sequence which is symmetric about a bad square
  corresponds to a symbolic trajectory of a discontinuity of the
  interval exchange.
\end{proof}

\section{Concluding Remarks}
We end by stating one simple consequence of
Theorem~\ref{thm:characterization}.  If the permutations $h$ and $v$
defining the square-tiled surface are such that
$h(\lambda) \neq v(\lambda)$ for each $\lambda \in \Lambda$, then each
cutting sequence is uniquely determined by the sequence of labels of
squares, $(\lambda_n)_{n \in \mathbb{N}}$, and the labels of edges can
be forgotten.  For each $\lambda_i$, the next symbol $\lambda_{i+1}$
will be either $v(\lambda_i)$ or $h(\lambda_i)$.  Since
$v(\lambda_i) \neq h(\lambda_i)$, the symbol $\lambda_{i+1}$ uniquely
determines the edge crossed by the corresponding geodesic.

As Vincent Delecroix remarked to the author, it may be possible that
the sequence of labels $(\lambda_n)_{n \in \mathbb{N}}$ uniquely
determines a cutting sequence $(\lambda_n, e_n)$ even if
$h(\lambda) = v(\lambda)$ for some $\lambda$.  On the other hand, this
is not always the case, as shown in the following example
due to Pat Hooper.  Suppose $\Lambda = \{1, 2, ..., d\}$ and let $h = v$ be the
permutation $(1 \, 2 \, 3 \, \cdots \, d)$.  The sequence of the
labels of edges crossed by any geodesic on the surface is the periodic
sequence $\left(h^n(1)\right)_{n \in \mathbb{Z}}$, and so the cutting
sequences are not completely determined by the labels of the squares.



\end{document}